\documentclass[10pt]{article}

\usepackage[latin2]{inputenc}
\usepackage{amsmath}
\usepackage{bbm}
\usepackage{graphicx}
\usepackage{amssymb}
\usepackage{esint}
\usepackage{tikz}
\usepackage[dvipsnames]{xcolor}
\usepackage{floatrow}
\usepackage{color}
\usepackage{amsthm}
\usepackage{stackengine}
\usepackage{scalerel}
\usepackage{epsfig}
\usepackage[english]{babel}
\usepackage{mathtools}
\newtheorem{theorem}{Theorem}
\newtheorem{proposition}[theorem]{Proposition}
\newtheorem{lemma}[theorem]{Lemma}
\newtheorem{corollary}[theorem]{Corollary}

\newtheorem{remark}[theorem]{Remark}

\usepackage[colorlinks=true,linkcolor=RoyalBlue,anchorcolor=black,citecolor=Green,filecolor=black,menucolor=black,runcolor=black,urlcolor=black]{hyperref}
 \usepackage[margin=1in]{geometry}
 \usepackage[page,toc,titletoc,title]{appendix}
 \usepackage{tocloft}
 \usepackage{accents}

    \providecommand{\PfStep}[2]{\noindent \ifnum\value{#1}=1\else\fi{\sc Step }\arabic{#1}\label{#2}\refstepcounter{#1}.} 

\newtheorem*{theorem*}{Theorem}

\def\Xint#1{\mathchoice
{\XXint\displaystyle\textstyle{#1}}%
{\XXint\textstyle\scriptstyle{#1}}%
{\XXint\scriptstyle\scriptscriptstyle{#1}}%
{\XXint\scriptscriptstyle\scriptscriptstyle{#1}}%
\!\int}
\def\XXint#1#2#3{{\setbox0=\hbox{$#1{#2#3}{\int}$ }
\vcenter{\hbox{$#2#3$ }}\kern-.6\wd0}}

\def\fint{\Xint-}

\allowdisplaybreaks[2]

\newcommand{\twoscale}{\textrm{2-scale}}
 
\newcommand{\supp}{\operatorname{supp}}

\newcommand{\diam}{\operatorname{diam}}
\newcommand{\eps}{\varepsilon}

\newcommand{\dist}{\operatorname{dist}}

\newcommand{\loc}{\rm{loc}}

\newcommand{\dd}{{\textrm{d}}}
\newcommand{\wrt}{w.\,r.\,t.\ }
\newcommand{\as}{a.\,s.\ }
\newcommand{\eg}{e.\,g.}
\newcommand{\ie}{i.\,e.\ }
\newcommand{\lhs}{left-hand side\ }
\newcommand{\wolog}{w.\,l.\,o.\,g\ }
\newcommand{\rhs}{right-hand side\ }

\newcommand{\domain}{{\mathcal{O}}}

\renewcommand{\vec}[1]{{\mathbf{#1}}}

\newcommand{\en}[1]{\left< #1 \right>} 
\providecommand{\dx}{\, \mathrm{d} x}
\providecommand{\dy}{\, \mathrm{d} y}
\providecommand{\dS}{\, \mathrm{d} S}
\providecommand{\ud}{\, \mathrm{d}}
\newcommand*{\Hd}{\mathbb{H}^d_+}
\newcommand{\Rd}{\mathbb{R}}
\newcommand{\C}{\mathcal{C}}

\usepackage{authblk} 	
\usepackage{textcomp}

\numberwithin{equation}{section} 

\newcommand{\transpose}{*}

\title{Regularity theorems for random elliptic operators on domains}

\author{
Peter~Bella$^a$,
Julian~Fischer$^b$,
Marc~Josien$^c$,
Claudia~Raithel$^d$
}

\date{}

\begin{document}

\maketitle

\begin{abstract}
Regularity theorems \`a la Avellaneda-Lin are an indispensable part of the modern quantitative theory of stochastic homogenization. While interior regularity results for random elliptic operators have been available for a while, on general smooth domains the existing theory has until recently remained limited to Lipschitz estimates. We establish $C^{1,\alpha}$ regularity results for random elliptic operators on bounded sufficiently smooth domains, as well as for scalar problems on convex polytopes. We, furthermore, prove a number of auxiliary results typically employed in the derivation of fluctuation bounds, such as a weighted Meyers estimate.
\end{abstract}

\tableofcontents

\thanks{
~\\
$^a$TU Dortmund, DE-44227, Dortmund, Germany
\\
$^b$IST Austria, AT-3400, Klosterneuburg, Austria
\\
$^c$CEA, DES, IRESNE, DEC, SESC, LMCP, Cadarache, F-13108, Saint-Paul-Lez-Durance, France
\\
$^d$TU Wien, AT-1040, Wien, Austria}

\section{Introduction}
In this contribution we derive regularity results for weak solutions $u\in H^1_{loc}(\domain)$ of elliptic PDEs and systems of the form
\begin{equation}
\label{equation_on_domain}
\begin{aligned}
-\nabla \cdot (a(\tfrac{\cdot}{\eps}) \nabla u) &= 0 && \text{in }\quad \, \domain, 
\\
u &= 0 && \text{on } \quad \partial \domain,
\end{aligned}
\end{equation}
where the coefficient-field $a: \mathbb{R}^d \rightarrow \Rd^{d \times d}$ is assumed to be bounded and uniformly elliptic --properties which are classically observed to be insufficient for guaranteeing the H\"older regularity of $u$, see for instance the counter-examples by Meyers \cite[Example 3]{PS} and De Giorgi \cite[Section 9.1.1]{GM_book}. Within the context of homogenization, in which the operator $-\nabla \cdot a(\tfrac{\cdot}{\eps}) \nabla$ is approximated by a constant-coefficient operator $-\nabla \cdot \bar a \nabla$ as $\eps \rightarrow 0$, it is, however, quite standard to aim at ``transferring'' regularity properties from the constant-coefficient homogenized operator back to  $-\nabla \cdot  a(\tfrac{\cdot}{\eps}) \nabla$ for small enough $\eps>0$. This type of strategy was introduced by Avellaneda and Lin in the setting of periodic homogenization in the 80s \cite{AvellanedaLinCPAM,AvellanedaLinCRAS89}, and more recently has been fruitfully exploited in stochastic homogenization \cite{armstrongsmart2014,ArmstrongMourrat,ArmstrongKuusiMourratNew,FischerOtto,GNO_final,
GNO5,GloriaOtto2015}. 

In the past decade, an extensive theory of quantitative stochastic homogenization has been developed. As large-scale regularity statements are a central ingredient in the derivation of quantitative homogenization estimates \cite{armstrongsmart2014,ArmstrongKuusiMourratNew,GNO_final,GloriaOtto2015}, at this point the topic of large-scale regularity theories for random elliptic operators on the whole-space has been quite thoroughly explored. We refer to, \eg, \cite{ArmstrongKuusiMourratNew, ArmstrongMourrat, armstrongsmart2014, GloriaOtto2015,GNO_final,MarahrensOtto} concerning H\"older, Lipschitz, and $C^{1,\alpha}$ regularity, to \cite{ArmstrongKuusiMourrat_higherorder,FischerOtto} for higher-order (that is, $C^{k,\alpha}$) regularity, as well as to
\cite{ArmstrongSmartNondivergence,AFK,BellaFehrmanOtto,ArmstrongDario,BCF,ABM} for nonlinear, degenerate, or parabolic variants.

On the other hand, much less is known concerning the topic of boundary regularity. In \cite{ArmstrongMourrat} large-scale Lipschitz estimates were shown for random elliptic operators with Dirichlet boundary data on sufficiently regular domains. For the particular case of a half-space $\domain=\Hd$, in \cite{FischerRaithel} a $C^{1,\alpha}$ regularity theory up to the boundary was developed. In \cite{Raithel} an analogous result was proven in the case of Neumann boundary data. However, for general smooth domains the derivation of a $C^{1,\alpha}$ regularity theory for random elliptic operators up to the boundary has remained an open problem.

The main result of the present work --stated in Theorem~\ref{large_scale_reg_domain}-- is a large-scale $C^{1,\alpha}$ regularity theory for random elliptic operators on a general smooth domain in the case of Dirichlet boundary conditions. To facilitate the analysis of numerical homogenization schemes, we also prove a corresponding result for scalar PDEs on convex polytopes, see Theorem~\ref{PropositionExcessDecayConvex}. In Proposition~\ref{sens_Lm0} we state a weighted Meyers estimate, which is employed in the companion paper \cite{BFJR_2} to derive fluctuation estimates on boundary layers. All of these results rely on a (suboptimal) quantitative estimate on the boundary layer corrector, stated in Proposition~\ref{intermediate_lemma}.

\subsection{Boundary layers in periodic and stochastic homogenization}

Two concepts central to quantitative homogenization theory are that of the \emph{homogenization corrector} and the \emph{two-scale expansion}. For the linear elliptic PDE
\begin{align}
\label{LinearEllipticRd}
-\nabla \cdot (a(\tfrac{\cdot}{\eps})\nabla u^\eps)=\nabla \cdot f \quad\quad\text{on }\mathbb{R}^d,
\end{align}
the homogenization corrector $\phi_i \in H^1_{\loc}(\Rd^d)$ is defined as a sublinearly growing solution to the corrector equation
\begin{equation}
\label{defn_whole_space_corrector}
\begin{aligned}
-\nabla \cdot (a \nabla (x_i + \phi_i)) &= 0 && \text{in }\quad \, \Rd^d.
\end{aligned}
\end{equation}
Note that its rescaling $\phi_i^\eps(\cdot):=\eps\phi_i(\tfrac{\cdot}{\eps})$ solves the rescaled corrector equation $-\nabla \cdot (a^\eps (e_i+\nabla \phi_i^\eps))=0$, where we have abbreviated $a^\eps(\cdot):=a(\tfrac{\cdot}{\eps})$.

For elliptic, stationary, and ergodic random fields $a$, the limit $\eps\rightarrow 0$ of the solutions $u^\eps$ to \eqref{LinearEllipticRd} is given by the solution $\bar u$ to a constant-coefficient \emph{effective equation} $-\nabla \cdot(\bar a\nabla \bar u)=\nabla \cdot f$ \cite{Kozlov,PapanicolaouVaradhan}. Here, the \emph{effective coefficient} $\bar a$ is determined as $\bar a e_i := \mathbb{E}[a(e_i+\nabla \phi_i)]$.

Given the solution to the homogenized problem $\bar u$, the \emph{two-scale expansion}
\begin{align}
\label{2_scale}
u^\eps_{\twoscale} := \bar u + \phi_i^\eps \partial_i \bar{u}
\end{align}
provides an approximate solution to the original PDE \eqref{LinearEllipticRd}.
Many estimates in quantitative homogenization (both periodic and stochastic) follow from the energy estimate for the divergence-form elliptic equation satisfied by the ``homogenization error'' $u^\eps - u^\eps_{\twoscale}$.

In the case of the problem on a bounded domain $\domain$ with Dirichlet boundary conditions, even with $u^\eps=\bar u$ on $\partial\domain$ the homogenization error fails to vanish on the boundary: In this case we have $u^\eps - u^\eps_{\twoscale} = -\phi_i^\eps \partial_i \bar u$ on $\partial\domain$, which is in general nonzero.

The avoid this mismatch in the boundary data of $u^\eps$ and $u^\eps_{\twoscale}$ it is standard (for classical literature in the periodic setting see, \eg, \cite{AllaireAmar, GerardVaretMasmoudi, GerardVaretMasmoudi2}) to adapt the ansatz $u^\eps_\twoscale$ via a boundary layer corrector $\theta^{\eps}_i \in H^1_{\loc}(\domain)$. The latter is given as the weak solution of
\begin{equation}
\label{defn_theta}
\begin{aligned}
-\nabla \cdot (a^{\eps} \nabla \theta^{\eps}_i) &= 0 && \text{in }\quad \, \domain, 
\\
\theta^{\eps}_i &= \phi_i(\cdot/ \eps) && \text{on } \quad \partial \domain.
\end{aligned}
\end{equation}
The adjusted two-scale expansion then becomes 
\begin{align}
\label{boundary_2_scale}
u^\eps_{\textrm{2-scale}, \partial} := \bar u + (\phi_i^\eps - \eps \theta_i^\eps) \partial_i \bar{u}.
\end{align}
The adapted ``homogenization error'', given by $u^\eps - u^\eps_{\twoscale, \partial}$, now has homogeneous Dirichlet boundary data and solves the elliptic PDE
\begin{equation}
\label{homogenized_error_domain}
\begin{aligned}
-\nabla \cdot (a^\eps \nabla (u^\eps - u^\eps_{\twoscale, \partial})) &= \nabla \cdot ((a^\eps (\phi_i^\eps - \eps \theta_i^\eps) - \sigma_i^\eps) \nabla \partial_i \bar u ) + a^\eps \eps \nabla \theta_i^\eps \cdot \nabla \partial_i \bar u  && \text{in }\quad \, \domain, 
\\
u^\eps - u^\eps_{\twoscale, \partial} &= 0 && \text{on } \quad \partial \domain.
\end{aligned}
\end{equation}
Here, $\sigma^{\eps}$ is the so-called (whole-space) flux-corrector; in particular, $\sigma^{\eps}$ is a 3-tensor that is skew-symmetric in its last two indices and solves $\nabla_k \cdot \sigma^{\eps}_{ijk} = e_j \cdot (a^{\eps}(e_i + \nabla \phi^{\eps}_i) - \bar{a} e_i)$ on $\Rd^d$ \cite[Section 7.2]{ZikovKozlovOlejnik}.
Given suitable estimates on the correctors $\phi_i$, $\sigma_i$, and $\theta_i^\eps$, the error equation \eqref{homogenized_error_domain} entails bounds on $u^\eps-u^\eps_{\twoscale,\partial}$. As no suitable estimates on $\theta_i^\eps$ are available in the literature, our first task (in Proposition~\ref{intermediate_lemma} below) is to derive suitable bounds for $\eps \theta_i^\eps$ and $\eps \nabla \theta_i^\eps$.

Before turning to the statement of our main results, we would like to more thoroughly position this contribution within the existing literature on boundary regularity in the context of homogenization of elliptic equations. As previously mentioned, Avellaneda and Lin
\cite{AvellanedaLinCPAM} have derived a $C^{0,1}$ regularity theory up to the boundary in the setting of periodic homogenization and Dirichlet boundary conditions.
The work of Armstrong and Mourrat \cite{ArmstrongMourrat} provides a large-scale $C^{0,1}$ theory up to the boundary for random linear elliptic operators with Dirichlet boundary data. In \cite{ArmstrongShen} the case of almost-periodic coefficients is handled. We remark that in \cite{FischerRaithel}, the second and fourth authors obtained a large-scale $C^{1,\alpha}$-theory for the Dirichlet problem with random coefficients on the half-space. This required the iterative construction of a sublinear corrector with Dirichlet boundary conditions, corresponding to $\phi^\eps - \eps \theta^\eps$. The fourth author extended the methods of \cite{FischerRaithel} to the Neumann problem \cite{Raithel}. Using similar techniques, the second and fourth authors considered materials with interfaces \cite{JosienRaithel} (see also, \eg, \cite{Zhuge,Josien,Blanc_LeBris_Lions} for the periodic case) and the Dirichlet problem in general angular sectors \cite{JosienRaithelSchaeffner}. 

We remark, furthermore, that one of the main motivations of the authors for the current contribution is the use of its results in the companion paper \cite{BFJR_2} --in which we derive optimal estimates for the decay of the boundary layer corrector $\theta^\eps$ in the case of the half-space $\domain = \Hd$. While \cite{BFJR_2} is the, to the best of our knowledge, first contribution to quantify the boundary layer in the stochastic setting, there is a rich literature on this topic in the periodic case (see, \eg, \cite{AllaireAmar,ArmstrongKuusiMourratPrange,GerardVaretMasmoudi,GerardVaretMasmoudi2,ShenZhugeRegularity}). In particular, in \cite{ArmstrongKuusiMourratPrange} optimal homogenization rates are obtained for the oscillating Dirichlet problem (of which \eqref{defn_theta} with periodic $a$ is an example) on uniformly convex domains, significantly improving the pioneering result of \cite{GerardVaretMasmoudi}. Note that the optimal result for the case $d=2$ in \cite{ArmstrongKuusiMourratPrange} relies on an improved regularity result for the homogenized boundary data established in \cite{ShenZhugeRegularity}. The authors of \cite{ShenZhugeRegularity} subsequently also quantified the boundary layer in the case of oscillating Neumann boundary data~\cite{ShenZhugeNeumann}.
  
\smallskip\smallskip\smallskip
\noindent
\textbf{Notation.}
Throughout this paper, we consider the half-space $\Hd := \{ x \in \Rd^d : x_1 >0 \}$. By $B_r(x)$ we denote the ball of radius $r$ centered at $x$ and $B_r^+(x) := B_r(x) \cap \Hd$. When the ball is centered at $0$ we simply write $B_r$ and $B_r^+$ respectively. The coordinate functions are given as $x_i: = x \cdot e_i$ for $i = 1, \ldots, d$. When $\domain = \Hd$ we decompose $x = (x^\perp, x^\parallel)$, where $x^\perp = x_1$ and $x^\parallel = (x_2, \ldots, x_d)$.

For any $\delta >0$, we will use the convention $\domain_\delta := \{ x \in \domain : \dist(x, \partial \domain) < \delta \}$. When $\domain$ is a $C^{1,1}$-domain (not necessarily bounded), there exists $c(\domain)$ small enough such that for any $x \in \domain_{c(\domain)}$ there exists a unique point $p(x) \in \partial\domain$ that is the closest point to $x$. Notice that when $\domain = \Hd$, for $x \in \Hd$ we have that $p(x)=x^\parallel$.

We frequently use the convention that $``\lesssim"$ means $`` \leq C(d, \lambda, \nu, \alpha),"$ where $d$ denotes the dimension, $\lambda>0$ is the ellipticity ratio of the coefficient field $a$, and $\nu$ and $\alpha$ are from the H\"older regularity assumption (A4) below.

We use the notation $\mathcal{C}$ to denote a generic random constant (respectively, $\mathcal{C}(x)$ to denote a generic random field) whose value may change from line to line. Unless otherwise specified, the constant $\mathcal{C}$ (respectively the random field $\mathcal{C}(x)$) will be subject to a uniform stretched exponential moment bound.  

Aside from the caveat that, for us, $H^1_{\loc}(\domain):= \left\{v : v \in H^1(B_r(x)\cap\domain) \text{ for all } r>0, x \in \domain  \right\}$, we use standard notation for Sobolev spaces. In particular, $\dot H^1(\domain)$ refers to the space of (in case of unbounded $\domain$, decaying) $L^2_{loc}$ functions with square-integrable distributional derivative.

Throughout the paper, we use the Einstein summation convention, i.\,e.\ we sum over repeated indices such as in $A_i (e_i+\nabla \phi_i)$.

\section{Discussion of main results}

\subsection{Assumptions}

\label{setting_discussion}

Throughout this contribution we make the following assumptions:

\paragraph{\textit{Assumptions on the domain.}} For simplicity, we always assume that $d \geq 3$. Then, in different sections, we have varying assumptions on our domain:

\begin{itemize}
\item In Sections \ref{Prop_1_proof}, \ref{operators_on_domain}, and \ref{Meyers}, unless otherwise stated, we assume that $\domain \subseteq \Rd^d$ is either a bounded $C^{1,1}$-domain or the half-space $\Hd$.
\item In Section \ref{regularity_near_corner}, we assume that $\domain \subseteq \Rd^d$ is either a convex cone or a convex polytope. 
\end{itemize}

\paragraph{\textit{Assumptions on the ensemble.}} We assume that the coefficient field, $a$, is defined on the whole-space and satisfies the following conditions:

\begin{itemize}
\item[(A1)] We assume that $\langle\cdot\rangle$-a.\,s. the heterogeneous coefficients $a$ are \emph{uniformly elliptic and bounded}. In particular, there exists $\lambda>0$ such that 
\begin{align*}
|a(x)v|&\leq |v|,
\\
a(x)v \cdot v &\geq \lambda |v|^2,
\end{align*}
for all $v\in \mathbb{R}^d$ and a.\,e.\ $x\in \mathbb{R}^d$.
\item[(A2)] The ensemble $\langle\cdot\rangle$ is \emph{stationary}. In particular, the law of $a$ is invariant under spatial shifts; i.\,e. the statistics of $a$ are spatially independent.
\item[(A3)] We assume that $\langle\cdot\rangle$ satisfies a quantitative ergodicity assumption in form of a \emph{spectral gap inequality} in the following form: For any random variable $\xi$ the estimate
\begin{align}
\label{spectral_gap_L2}
\en{ (\xi - \en{\xi} )^2} \leq \en{ \int_{\Rd^d}  \Big( \fint_{B_{1}(x)} \big| \frac{\partial \xi}{ \partial a}  \big| \dd y \Big)^2  \dx}
\end{align}
holds. Here, $\fint_{B_{1}(x)} \big| \frac{\partial \xi}{ \partial a}  \big| \dd y$  is given as
\begin{align}
\label{deriv}
\sup_{\delta a}\limsup_{h \rightarrow 0 } \frac{\xi(a + h \delta a ) - \xi(a)}{h} 
\end{align}
 with the perturbations $\delta a: \Rd^d \rightarrow \mathbb{R}^{d \times d}$ being supported in $B_{1}(x)$ and satisfying $\| \delta a \|_{L^{\infty}(\Rd^d)} \leq 1$. When the random variable $\xi$ is Fr\'echet differentiable then \eqref{deriv} is $\big|\frac{\partial \xi}{\partial a(x)}\big|$. 
\item[(A4)] For $\alpha \in (0,1]$, we assume that $a$ is H\"older continuous with stretched exponential moments; \ie there exist $\nu >0$ and a stationary random field $\mathcal{C}(a,x)$ satisfying 
\begin{align}
\label{stretched}
\mathbb{E}[\exp(\nu \mathcal{C}(a,x)^\nu)]\leq 2
\end{align}
such that the H\"older bound
\begin{align}
\label{holder_norm}
\sup_{y,z\in B_{1}(x)}  \frac{|a(y)-a(z)|}{|y-z|^\alpha} \leq \mathcal{C}(a,x)
\end{align}
holds. Notice that the expectation on the \lhs of \eqref{stretched} does not depend on the point $x \in \mathbb{R}^d$ due to stationarity.  
\end{itemize}
\noindent
The assumptions (A1)-(A3) are a standard set of assumptions in stochastic homogenization --with the quantification of ergodicity (the qualitative assumption of which is canon in classical qualitative stochastic homogenization) via the spectral gap inequality (A3) being particularly amenable to the methods of \cite{GNO_final, GNO5,GO1}. The assumption (A4) gives us access to (local) Schauder estimates --avoiding classical counterexamples to regularity at small scales--, whereas large-scale regularity properties are derived from the homogenization process. Assumptions (A1)-(A3) will be required for all of our main results, while assumption (A4) will only be needed for certain pointwise statements.

Since our assumptions on the ensemble are relatively standard, they are satisfied by a variety of relevant ensembles --\eg, (local nonlinear functions of) Gaussian ensembles satisfying suitably strong decorrelation estimates. To avoid excessive repetition, we choose to omit the details here and instead refer the interested reader to, \eg, \cite[Section 2.1]{BFJR_2}.

\begin{remark}[Scalar PDEs versus systems]
While for notational convenience we employ scalar notation for the statement and the proofs of our results, with the exception of Theorem~\ref{PropositionExcessDecayConvex} our proofs also work in the case of linear elliptic systems and thus our results also apply \emph{mutatis mutandis} to the case of linear elliptic systems.
\end{remark}

\begin{remark}[The case $d=2$] Notice that we have avoided the case $d=2$. This is, in particular, for brevity --all of our results still hold for $d=2$, up to some logarithmic corrections.
\end{remark}

\subsection{A suboptimal decay estimate for the boundary layer corrector}

In order to prove a large-scale regularity result for random $a^\eps$-harmonic operators on bounded $C^{1,1}$-domains with homogeneous Dirichlet boundary data, we first require access to a suboptimal estimate for the decay of the boundary layer $\theta^\eps$ away from $\partial \domain$. Since in \cite{BFJR_2} we require this suboptimal estimate as an ingredient for obtaining the optimal decay for the boundary layer in the case $\domain = \Hd$, we also state and prove our proposition for $\domain = \Hd$.
\begin{proposition}
\label{intermediate_lemma} Adopt the assumptions (A1)-(A4), let $0<\eps\leq 1$, and let $\domain\subseteq \Rd^d$ be either a bounded $C^{1,1}$-domain or $\mathbb{H}^d_+$.
Let $\theta^\eps\in H^1_{loc}(\domain)$ denote the unique weak solution to \eqref{defn_theta} with $\lim_{r\rightarrow\infty} \fint_{B_r} \chi_\domain |\nabla \theta^\eps|^2 \dx =0$.
Then there exists a random field $\C(a^\eps,x)$ such that $\frac{\C(a^\eps,x)}{\eps}$ has stretched exponential moments (in the sense of \eqref{stretched}, uniformly in $\eps$ and $x$) and the estimate
\begin{align}
\label{intermediate}
|\eps \nabla \theta^{\eps} (x_0)|   \lesssim_{d, \lambda, \nu, \alpha} \mathcal{C}(a^\eps,x_0)\Big(1 + \frac{\dist(x_0,\partial\domain)}{\varepsilon}\Big)^{-\frac{1}{3}}
\end{align}
holds for any $x_0 \in \domain$.

In case of $\domain=\Hd$, the random field $\C(a^\eps,x)$ is additionally stationary \wrt shifts tangential to $\partial \Hd$.
\end{proposition}

\noindent The proof of Proposition \ref{intermediate_lemma} is given in Section \ref{Prop_1_proof} and, while relying on stochastic ingredients from \cite{GNO_final,GNO5}, is entirely deterministic here. 

The main step for proving Proposition \ref{intermediate_lemma} is Lemma \ref{Prop_1_aux}, which is a spatially averaged decay estimate for the boundary layer corrector --the argument for which relies on a homogenization infused Campanato iteration to compare the boundary layer with a homogenized counterpart at large-scales. As already mentioned above, this is a standard technique in both periodic and stochastic homogenization. 

We remark that the spatially averaged decay estimate of Lemma \ref{Prop_1_aux} holds under the assumptions (A1)-(A3). In order to obtain \eqref{intermediate}, the result of Lemma \ref{Prop_1_aux} must be post-processed under the additional inclusion of assumption (A4).

\subsection{Regularity for random $a^\eps$-harmonic functions on regular domains}

Using the suboptimal decay of the boundary layer from Proposition \ref{intermediate_lemma}, we are then able to obtain the following large-scale regularity result for random $a^\eps$-harmonic functions with homogeneous Dirichlet 
boundary data on bounded $C^{1,1}$-domains or on $\Hd$.
Note that like the original Avellaneda-Lin regularity theorems \cite{AvellanedaLinCPAM,AvellanedaLinCRAS89} and in contrast to standard Schauder theory (the latter of which is applicable on small scales $\lesssim \eps$, assuming a H\"older regular coefficient field $a$), Theorem~\ref{large_scale_reg_domain} establishes regularity results for the operator $-\nabla \cdot a(\tfrac{\cdot}{\eps}) \nabla$ on scales $\gg \eps$.

\begin{theorem} \label{large_scale_reg_domain} Adopt the assumptions (A1) - (A3), let $0<\eps\leq 1$, and let $\domain\subseteq \Rd^d$ be either a bounded $C^{1,1}$-domain or $\Hd$. Then there exists a random field $r_{\domain}^*(a^\eps,x) \geq \varepsilon$ such that $\frac{r_{\domain}^*(a^\eps,x)}{\varepsilon}$ has stretched exponential moments (in the sense of \eqref{stretched}, uniformly in $\eps$ and $x$) and such that the following holds true:

Given any $0<R<c(\domain) $ and any $x_0 \in \overline{\domain}$, let $u \in H^1_{\loc}(\domain)$ solve 
\begin{equation}
\label{u_with_massive_term}
\begin{aligned}
-\nabla \cdot (a^\eps\nabla u) &= 0 && \text{in } \quad  B_R(x_0) \cap \domain ,\\
u & = 0 && \text{on } \quad  B_R(x_0) \cap \partial \domain.
\end{aligned}
\end{equation}
Then for any $r$ with $r_{\domain}^*(a^\eps,x_0) \leq r \leq R$, the excess-decay property 
\begin{align}
\label{excess_decay_thm}
& \inf_{A\in \mathbb{R}^d} \fint_{\domain \cap B_r(x_0)} |\nabla u^\eps - A \cdot \nabla ( x + \phi^\eps - \eps \theta^\eps) |^2 \dx
  \lesssim_{d, \lambda, \alpha, \domain} \bigg( \Big(\frac{r}{R}\Big)^{2\alpha} + \frac{r^2}{c(\domain)}\bigg)
\fint_{\domain \cap B_R(x_0)} |\nabla u^\eps|^2 \dx
\end{align}
holds.
Furthermore, for any $r$ with $r_{\domain}^*(a^\eps,x_0) \leq r \leq R$ the large-scale mean-value property 
\begin{align}
\label{large_scale_mvp_domain}
\fint_{\domain \cap B_r(x_0)} |\nabla u|^2 \dx  \lesssim_{d, \lambda, \domain} \fint_{\domain \cap B_R(x_0)} |\nabla u|^2 \dx
\end{align}
holds.

In case of $\domain=\Hd$, it holds that $c(\domain) = \infty$ and the random field $r_{\domain}^*$ is additionally stationary \wrt shifts tangential to $\partial \Hd$.
\end{theorem}
\noindent
We recall that in the particular case of the half-space $\domain=\Hd$, Theorem~\ref{large_scale_reg_domain} was already established in \cite{FischerRaithel}, though without explicit quantification of the minimal radius $r_\domain^*(a^\eps,x_0)$.
The argument for Theorem \ref{large_scale_reg_domain}, contained in Section \ref{operators_on_domain}, is similar to that of Lemma \ref{Prop_1_aux}. It also requires the result of Lemma \ref{Prop_1_aux}, a spatially averaged version of Proposition~\ref{intermediate_lemma}, as an input.

Combining the large-scale regularity of Theorem \ref{large_scale_reg_domain} with classical Schauder estimates at small scales (available to us by additionally imposing assumption (A4)), we obtain two useful corollaries.

\begin{corollary}
\label{cor_Reg06}
Suppose that assumptions (A1)--(A4) are satisfied, let $0<\eps\leq 1$, and let $\domain\subseteq \Rd^d$ be either a bounded $C^{1,1}$-domain or $\Hd$. Then there exists a random field $\mathcal{C}(a^\eps,x)$ with stretched exponential moments (in the sense of \eqref{stretched}, uniformly in $\eps$ and $x$) such that for any $0<r<c(\domain)$ and for any $x_0\in \overline{\domain}$ the following is true: Any weak solution $u \in H^1(\domain\cap B_r(x_0))$ of the PDE
\begin{equation*}
\begin{aligned}
-\nabla \cdot (a^\eps\nabla u) &= 0
&&\text{in }\domain\cap B_r(x_0),
\\
u&=0
&&\text{on }\partial\domain\cap B_r(x_0),
\end{aligned}
\end{equation*}
satisfies the estimate
\begin{equation}
\label{Reg06}
|\nabla u(x_0)|  \leq \mathcal{C}(a^\eps,x_0) \Big( \fint_{B_r(x_0)\cap \domain} |\nabla u|^2 \dx \Big)^{\frac12}.
\end{equation}
In case of $\domain=\Hd$, the random field $\C(a^\eps,x)$ is additionally stationary \wrt shifts tangential to $\partial \Hd$ and it holds that $c(\domain)=\infty$.
\end{corollary}
\noindent
We emphasize that the key contribution of Corollary~\ref{cor_Reg06} is that the estimate \eqref{Reg06} is uniform in $\eps$; for bounded domains and for $a$ satisfying (A4), an estimate of the form \eqref{Reg06} but with $\eps$-dependent constant is an immediate consequence of standard Schauder theory.

\begin{corollary}
\label{GradientBoundByIntegral}
Adopt the assumptions (A1)-(A4), let $0<\eps\leq 1$, and let $\domain\subseteq \Rd^d$ be either a bounded $C^{1,1}$-domain or $\mathbb{H}^d_+$. Let $u \in \dot H^1_0(\domain)$ be a weak solution of
\begin{equation*}
\begin{aligned}
-\nabla \cdot (a^\eps\nabla u) &= \nabla \cdot g + f
&&\text{in }\domain,
\\
u&=0
&&\text{on }\partial\domain.
\end{aligned}
\end{equation*}
Then there exists a random field $\mathcal{C}(a^\eps,x)$ with stretched exponential moments (in the sense of \eqref{stretched}, uniformly in $\eps$ and $x$) such that the bound
\begin{align}
\label{BoundSolutionRandomOperator}
|\nabla u(x_0)| 
\leq&
\mathcal{C}(a^\eps,x_0)
\int_{\domain} \frac{|g| + |f| \dist(x,\partial\domain)}{| x_0- x|^d} \dx
\end{align}
holds for all $x_0 \in \domain$. 

In case of $\domain=\Hd$, the random field $\C(a^\eps,x)$ is additionally stationary \wrt shifts tangential to $\partial \Hd$.
\end{corollary}

\noindent Notice that the \rhs of \eqref{BoundSolutionRandomOperator} is infinite whenever $x_0$ is a Lebesgue point for either $f$ or $g$ and if $f(x_0)\neq 0$ respectively $g(x_0)\neq 0$ holds.

\subsection{A large-scale weighted Meyers estimate}


As an important application of the mean value property \eqref{large_scale_mvp_domain} established in Theorem~\ref{large_scale_reg_domain}, we next prove a large-scale weighted Meyers estimate for random $a^\eps$-harmonic operators on either $C^{1,1}$-domains or $\Hd$. Introduced by Gloria, Neukamm, and Otto \cite[Step 3 of Proposition 3]{GNO_reg_3}, estimates of the below type have become standard tools in estimating fluctuations in quantitative stochastic homogenization of elliptic PDEs (see, \eg, also \cite[Lemma 4.3]{BellaFehrmanFischerOtto} or \cite[Lemma 44]{FN_2020}). Since, to the best of our knowledge, such results have not been proven on regular bounded domains or on the half-space, we provide them here.
\begin{proposition}[Weighted Meyers Estimate]
\label{sens_Lm0}
 Adopt the assumptions (A1) - (A3), let $0<\eps\leq 1$, and let $\domain\subseteq \Rd^d$ be either a bounded $C^{1,1}$-domain or $\Hd$. Furthermore, let $x_0 \in \domain$ and $R \geq \max\{r_{\domain}^*(a^\eps,x_0),r_{\domain}^*((a^\eps)^\transpose,x_0)\}$ (where $r_\domain^*$ is defined in Theorem \ref{large_scale_reg_domain}). For $f \in L^2(\domain)$ and $g \in L^2(\domain;\mathbb{R}^d)$, let $v\in \dot H^1_0(\domain)$ denote the unique weak solution of 
\begin{equation}
\label{EquationV}
\begin{aligned}
-\nabla \cdot (a^\transpose(\tfrac{\cdot}{\eps}) \nabla v ) &= \nabla \cdot g + f \quad &&\text{in }\domain,\\
v &= 0 \quad &&\text{on }\partial\domain,
\end{aligned}
\end{equation}
and introduce the weight
\begin{equation}
\label{omega}
 \omega_{\alpha,R}(x) := \Big( \frac{|x-x_0|}{R} + 1\Big)^\alpha.
 \end{equation}
Then, in the case that $f=0$, there exists $\bar p(\lambda, d, \domain) > 1$ such that, for $1 \le p < \bar p$ and $0 \le \alpha_0 < \alpha_1 < d(2p-1)$, the estimate
\begin{align}
\label{Lemma_4_est_1}
\begin{split}
&\Big( \int_{\domain}  |\nabla v|^{2p} \omega_{\alpha_0,R}  \dx \Big)^{\frac{1}{2p}}  \lesssim_{d,\lambda,p, \domain,\alpha_0, \alpha_1} 
\Big( \int_{\domain} |g|^{2p}  \omega_{\alpha_1,R}  \dx \Big)^{\frac{1}{2p}} 
\end{split}
\end{align}
holds. Likewise, in the case that $g = 0$, there exists $\bar p(\lambda, d, \domain) > 1$ such that, for $1 \le p < \bar p$ and $0 \le \alpha_0 < \alpha_1 -2p <\alpha_1 < d(2p-1)$, the estimate
\begin{align}
\label{Lemma_4_est_2}
\begin{split}
&\Big( \int_{\domain} |\nabla v|^{2p}  \omega_{\alpha_0,R}  \dx \Big)^{\frac{1}{2p}} \lesssim_{d,\lambda,p, \domain,\alpha_0, \alpha_1}
R \Big( \int_{\domain} | f|^{2p}  \omega_{\alpha_1,R}  \dx \Big)^{\frac{1}{2p}}
\end{split}
\end{align}
holds. 
\end{proposition}
\noindent
Proposition~\ref{sens_Lm0} is an important ingredient for the proof of the fluctuation estimates for the boundary layer corrector in the companion paper \cite[Proposition~7]{BFJR_2}.

The proof of Proposition \ref{sens_Lm0}, which is contained in Section~\ref{Meyers}, is quite similar to the proof of analogous results on the whole-space. It, in particular, depends on the use of the classical Meyers estimate on the contributions of a dyadic decomposition of $\domain$ in combination with large-scale regularity results for the random elliptic operator. 

\subsection{Regularity for $a^\eps$-harmonic functions on convex polytopes}

For the last result of this paper, we leave the setting of regular domains and consider instead the situation in convex polytopes.
The desire for a $C^{1,\alpha}$ regularity theory also in this setting is quite natural --in particular, in practice material samples are not always smooth and often have corners and edges. Furthermore, convex polytopes make a natural appearance in numerical homogenization:
Representative volume elements (RVEs) for determining the effective properties of random media are often chosen as cubes; numerical homogenization approaches such as the local orthogonal decomposition method (LOD method) by M\aa{}lqvist and Peterseim \cite{MalqvistPeterseim} also often utilizes cubes as domains of definition of the local basis functions. We refer to \cite{FischerGallistlPeterseim,HauckMohrPeterseim} for the analysis of the LOD method in the context of stochastic homogenization.

The issue separating the situation in a smooth domain from that in a general polytope is a lack of regularity in the homogenized picture. In particular, in domains with corners and edges it is standard practice to decompose a harmonic function in terms of a regular contribution and singular corner and edge contributions --an excess-decay only holding for the regular contribution (see, \eg, \cite{Dauge_book, Dauge_paper_1, Grisvard_book}). In this paper, using the explicit scaling of the corner and edge contributions (which depends on the opening angles), we first observe that in the convex situation an excess-decay in fact holds for the full harmonic function --this is in contrast, \eg, to the situation encountered in \cite{JosienRaithelSchaeffner} in which general opening angles ($<2\pi$) are considered. The presence of regularity in the homogenized setting then makes it possible to transfer regularity to $a^\eps$-harmonic functions at large scales. 
Imposing in addition our assumption (A4) to ensure regularity on small scales, we obtain the following pointwise estimate.
\begin{theorem}
\label{PropositionExcessDecayConvex} Adopt the assumptions (A1)-(A4), let $0<\eps\leq 1$, and let $\domain \subseteq \mathbb{R}^d$ be a convex polytope. Denote by $E$ the set of $(d-2)$-dimensional hyperedges of $\partial \domain$.

Then there exists a constant $\delta=\delta(\lambda,\domain)$ and a random field $r^*_{\domain}(a^\eps, x)\geq \eps$ such that $\frac{r^*_{\domain}(a^\eps, x)}{\eps}$ has stretched exponential moments (in the sense of \eqref{stretched}, uniformly in $\eps$ and $x$) and such that the following holds true: Given any $x_0\in \domain$, any $\rho$ with $r^*_{\domain}(a^\eps, x_0)\leq \rho\leq c(\domain)$, and any weak solution $u\in H^1(B_\rho(x_0))$ to the scalar elliptic PDE
\begin{equation*}
\begin{aligned}
-\nabla \cdot (a^\eps\nabla u) &= 0
&&\text{in }\domain\cap B_\rho(x_0),
\\
u&=0
&&\text{on }\partial\domain\cap B_\rho(x_0),
\end{aligned}
\end{equation*}
the estimate
\begin{align}
\label{BoundAHarmonicFirstDerivative}
|\nabla u(x_0)| \leq C \Big(\frac{\dist(x_0,E)}{\rho}\Big)^\delta \Big(\fint_{\domain\cap B_{\rho}(x_0)} |\nabla u|^2 \dx\Big)^{1/2}
\end{align}
holds.
\end{theorem}
\noindent
In other words, the theorem states that the gradient of any $a^\eps$-harmonic function must decay whenever we approach a $(d-2)$-dimensional hyperedge of the polytope. The underlying reason is that in the homogenized picture, at a $(d-2)$-dimensional hyperedge of the polytope there is no nonzero linear polynomial satisfying the homogeneous Dirichlet boundary conditions in a neighborhood of the hyperedge, nor is there a singular harmonic function with such homogeneous Dirichlet data. Note that, in contrast, near any point on a face of the polytope the distance function $\dist(\cdot,\partial\domain)|_\domain$ is a linear polynomial with Dirichlet boundary conditions, which is why we do not expect a gradient decay near points on the faces of the polytope.

\section{Argument for Proposition \ref{intermediate_lemma}: A suboptimal decay estimate for the boundary layer}
\label{Prop_1_proof}

The main step towards proving Proposition \ref{intermediate_lemma} is the following lemma: 

\begin{lemma}
\label{Prop_1_aux}  Under the assumptions of Proposition \ref{intermediate_lemma}, there exist random fields $\C(a,x)$ and $\tilde {\mathcal{C}}(a,x)$ with stretched exponential moments (in the sense of \eqref{stretched}) such that 
\begin{align}
\label{lemma_3_eqn}
\fint_{B_r(x_0) \cap \domain} |\eps \nabla \theta^{\eps}|^2 \dx \leq \C(a,x_0) \Big(1+\frac{r}{\varepsilon}\Big)^{-\frac{2}{3}}
\end{align}
holds for any $x_0 \in \partial \domain$ and any $r$ with $c(\domain) > r \geq \tilde {\mathcal{C}}(a, x_0) \varepsilon$.

If $\domain = \Hd$, then $\C(a,x)$ and $\tilde{C}(a, x)$ are stationary \wrt shifts tangential to $\partial \Hd$ and $c(\domain)=\infty$. 
\end{lemma}
\noindent As already mentioned, to obtain \eqref{lemma_3_eqn} we morally ``transfer the regularity'' from a corresponding homogenized solution. 


The argument for Lemma \ref{Prop_1_aux} requires a couple of auxiliary lemmas. The first auxiliary lemma, which is proven in Section \ref{Prop_1_auxil}, is a weighted energy estimate that is a consequence of the classical Hardy inequality. 
\begin{lemma}
\label{Hardy}
Let the assumption (A1) hold and let $x_0 \in \partial \mathbb{H}^d_+$. Then there exists $\kappa = \kappa(d,\lambda) > 0$ such that, for any $r>0$ and for any weak solution $u \in H^1(B^+_r(x_0))$ of 
\begin{equation*}
\begin{aligned}
-\nabla \cdot (a\nabla u) &= \nabla \cdot g + f&&\text{in } B^+_r(x_0), \\
u&=0&&\text{on } \partial B_r^+(x_0),
\end{aligned}
\end{equation*}
there $\langle \cdot \rangle$-\as holds 
\begin{align}
\label{weighted_energy}
 \int_{B^+_r(x_0)} \Big(1-\frac{|x - x_0|}{r}\Big)^\kappa |\nabla u|^2 \dx \lesssim_{d,\lambda, \kappa} \int_{B^+_r(x_0)} \Big(1-\frac{|x-x_0|}{r}\Big)^\kappa (|g|^2+r^2f^2) \dx.
\end{align}
\end{lemma}

\smallskip

The second auxiliary lemma gives the constant coefficient regularity result that we ``transfer'' onto the heterogenous solution.
\begin{lemma}
\label{constant_coeff_reg}
Let $x_0 \in \overline{\mathbb{H}}^d_+$. Furthermore, let $\bar{a}  \in \Rd^{d\times d}$ be uniformly elliptic and bounded. Let $v$ be $\bar{a}$-harmonic in $B^+_R(x_0)$ with homogeneous Dirichlet boundary data on $\partial \Hd \cap B_R^+(x_0)$. Then, for any $\rho \in (0,\frac{R}{2}]$, we have that 
\begin{align}
\label{constant_reg_1} 
\sup_{B^+_{R- \rho} (x_0) } (\rho^2 |\nabla^2 v|^2 + |\nabla v|^2) \lesssim_{d, \lambda}  \Big( \frac{R}{\rho}\Big)^d \fint_{B^+_R(x_0) } |\nabla v|^2 \dx.
\end{align}
\end{lemma}
\noindent This lemma follows from using the Sobolev embedding and iterating the Caccoppoli inequality in directions tangential to $\partial \Hd$. To recover the desired estimate in the $e_1$-direction, one uses the equation solved by $v$. The full proof is standard and can e.\,g.\ be found in \cite[Lemma 4.2]{FischerRaithel}, for brevity we do not repeat it here.

We remark that both Lemmas \ref{Hardy} and \ref{constant_coeff_reg} are only stated and proven on $\Hd$. In particular, in the case that $\domain$ is a bounded $C^{1,1}$-domain, these auxiliary results are applied only after we have straightened the boundary --which introduces technical details that must be handled in our argument.

The third auxiliary result which we require in our argument for Lemma \ref{Prop_1_aux} is the following bound on the boundary layer:

\begin{lemma} \label{exp_2}
Under the assumptions of Proposition \ref{intermediate_lemma} with $\domain = \Hd$, $\theta^{\eps}_i$ satisfies
\begin{align}
\label{int_along_rays}
\left \langle \int_0^{\infty} |\eps\nabla \theta^{\eps}_i (x_0 + z)|^2 \textrm{d} z \right \rangle \lesssim \eps
\end{align}
for any $x_0 \in \partial \Hd$. 
\end{lemma}

\smallskip

\noindent We remark that Lemma \ref{exp_2} is needed to initiate the Campanato iteration which is the backbone of the argument for Lemma \ref{Prop_1_aux} in the case that $\domain = \Hd$. Since, given our treatment of the exponentially localized boundary layer $\theta^{T}$ in Appendix \ref{exp_sec}, the proof is very short, we simply give it here:

\smallskip

\begin{proof}
Since we are now only on the half-space, we may set $\eps = 1$. Observe that it suffices to show (with $\theta_i^T$ solving \eqref{defn_theta_T})
\begin{align}
\label{IntLocalized}
\left \langle \int_0^{\infty} |\nabla \theta_i^T (x_0 + z)|^2 + \frac{1}{T} |\theta_i^T(x_0+z)|^2 \dd z \right \rangle \lesssim 1,
\end{align}
then \eqref{int_along_rays} follows by passing to the limit $T\rightarrow \infty$. The equation \eqref{IntLocalized} however is a straightforward consequence of taking the expected value in \eqref{exp_weighted_energy_estimate} for $g=\zeta \phi_i$ --$\zeta$ being a standard cutoff being equal to one on $\partial \Hd$ and zero outside of $\partial \Hd+B_1$. Using stationarity (recall that for simplicity we have assumed $d\geq 3$) as well as the standard whole-space corrector bounds (see Theorem~\ref{CorrectorBoundGNO4} in the appendix, from which $\fint_{B_r(y)} |e_i+\nabla \phi_i|^2 \dx \leq \mathcal{C}(a,y)$ follows by the Caccioppoli inequality), and passing to the limit $\gamma\rightarrow 0$ then yields the claim.
\end{proof}

In order to obtain Proposition \ref{intermediate_lemma}, once we have access to Lemma \ref{Prop_1_aux}, the estimate \eqref{lemma_3_eqn} must still be post-processed  to obtain the pointwise estimate \eqref{intermediate}. As previously indicated, this requires the additional assumption (A4) and then is a simple consequence of standard Schauder theory.

\subsection{Proof of Proposition \ref{intermediate_lemma}}

The main ingredient in the proof of Proposition \ref{intermediate_lemma} is Lemma \ref{Prop_1_aux}, which we prove below. 

\begin{proof}[Proof of Lemma~\ref{Prop_1_aux}] Throughout our argument we use $\mathcal{C}(a, x)$ to denote a generic random field --by tracking the random fields in the proof below, one finds that all of the random fields that appear are polynomial in the random field in the whole-space corrector bounds from Theorem~\ref{CorrectorBoundGNO4}. This, in particular, implies that they all have stretched exponential moments, and, in the case that $\domain = \Hd$, are stationary \wrt to shifts tangential to the boundary. 

Our argument comes in four steps: First, in Step 1, we locally straighten the boundary. Then, in Step 2, we initialize the Campanato iteration. In Step 3, we show the required excess decay --in particular, here we encounter some technicalities associated with straightening the boundary. Finally, in Step 4, we conclude our argument by iterating the excess decay estimate obtained in Step 3.\\

 \noindent Letting $1 \leq i \leq d$, here is the argument:\\
 
 \smallskip

\noindent{\bf Step 1: Straightening the boundary}  \qquad We assume that $c(\domain)$ is chosen small enough so that for all $c(\domain) \geq r >0$ we have that $B_r(x_0) \cap \domain$ is in one chart. We then let $\gamma: y \in \gamma^{-1}(B_r(x_0) \cap \domain) \mapsto x = \gamma(y) \in B_r(x_0) \cap \domain$ be $C^{1,1}$ and satisfy $\gamma(0)=x_0$ as well as $D\gamma(0)e_1=-\vec{n}(x_0)$ (with $\vec{n}$ denoting the exterior unit normal of $\domain$). Without loss of generality and to simplify notation, we may assume $D\gamma(0)=\operatorname{Id}_{d\times d}$ (and therefore in particular $e_1=-\vec{n}(x_0)$). Since $\gamma$ is $C^{1,1}$, for an arbitrary but fixed $\delta>0$ we may furthermore assume that the constant $c(\domain)$ is small enough so that for any $r\leq c(\domain)$ we have $B_{(1-\delta)r} (x_0) \cap \domain \subseteq \gamma (B_r^+) \subseteq B_{(1+\delta)r}(x_0) \cap \domain$.

As already mentioned, we use the notational convention that for any $f: B_r (x_0) \cap \domain \rightarrow \mathbb{R}^d$ we abbreviate $f(\gamma(\cdot))$ as $\tilde{f} (\cdot):= f(\gamma(\cdot))$; furthermore, for $b: B_r (x_0) \cap \domain \rightarrow \mathbb{R}^{d\times d}$, we write 
\begin{align}
\label{flattened_coeff_general}
\begin{split}
\hat{b}_{\mu \nu} (\cdot) := & \,   |\det(D \gamma)| (\cdot) \Big( \frac{\partial \gamma_{\mu}^{-1}}{\partial x_{\alpha}} b_{\alpha \beta}  \frac{\partial \gamma_{\nu}^{-1}}{\partial x_{\beta}} \Big)(\gamma (\cdot)) \\
= & \, |\det(D \gamma)|  \stackon[-8pt]{$ \frac{\partial \gamma_{\mu}^{-1}}{\partial x_{\alpha}} b_{\alpha\beta}  \frac{\partial \gamma_{\nu}^{-1}}{\partial x_{\beta}}$}{\vstretch{1.4}{\hstretch{2}{\widetilde{\phantom{\;\;\;\;\;\;\;\;}}}}}(\cdot).
\end{split}
\end{align}
We remark that $\hat{b}$ is uniformly elliptic and bounded as long as $a$ is such. Using the above conventions, for any $r \leq c(\domain)$, we notice that $\widetilde{\theta}^{\eps}$ solves
\begin{equation}
\label{theta_flattened_boundary}
\begin{aligned}
-\nabla \cdot (\widehat{a^{\eps}} \nabla \widetilde{\theta^{\eps}_i}) &= 0 && \text{in } \quad  B_{(1-\delta)r}^+,\\
\widetilde{\theta^{\eps}_i} & = \widetilde{\phi_i(\cdot / \eps)} && \text{on } \quad \partial \Hd\cap B_{(1-\delta)r}.
\end{aligned}
\end{equation} 

Notice that in the case that $\domain$ is any half-space, $\gamma$ may be taken to be simply a rotation and/or translation.

\bigskip

\noindent {\bf Step 2: Initiation of the Campanato iteration.} \qquad We treat the cases when $\domain = \Hd$ and $\domain$ is a bounded $C^{1,1}$-domain separately.  

\vspace{.2cm}

\noindent \emph{Case of a bounded $C^{1,1}$-domain:} \qquad In this case, we use the energy estimate for the equation satisfied by $\theta^{\eps}_i - \zeta \phi_i(\cdot / \eps)$ on $\domain$, where $\zeta$ is a smooth cut-off of width $\rho$ for the boundary $\partial \domain$. In particular, $\eta = 1$ on $\partial \domain$ and vanishes outside of $\domain_{\rho}:=\domain\cap (\partial \domain +B_\rho)$, which we use to denote the layer of width $\rho$ around $\partial \domain$. Then we have
\begin{align}
\label{initiation_bounded_domain}
\begin{split}
\int_{\domain} |\nabla (\theta^{\eps}_i- \zeta \phi_i(\cdot/ \eps))|^2\dx & \lesssim \int_{\domain_{\rho}} |\nabla (\zeta \phi_i(\cdot/ \eps))|^2\dx\\
& \lesssim \rho^{-2} \int_{\domain_{\rho}} |\phi_i(\cdot/ \eps)|^2 \dx + \eps^{-2} \int_{\domain_{\rho}} |\nabla \phi_i (\cdot/ \eps) |^2 \dx \leq \C(a) \Big(\rho^{-1} + \frac{\rho}{\eps^2}\Big).
\end{split}
\end{align}
Notice that here we have used the existing corrector bounds from the literature, see Theorem~\ref{CorrectorBoundGNO4}. We then deduce the bound
\begin{align}
\label{initiation_bounded_domain_2}
\begin{split}
\int_{\domain} |\eps \nabla \theta^{\eps}_i|^2\dx \leq \C(a) \Big(\frac{\eps^2}{\rho} + \rho\Big).
\end{split}
\end{align}
Setting $\rho = \eps$ yields the estimate
\begin{align}
\label{FirstBLCEstimateDomain}
\int_{\domain} |\eps \nabla \theta^{\eps}_i|^2\dx \leq \mathcal{C}(a) \eps.
\end{align}
\\

\noindent \emph{Case of $\Hd$:} \qquad  We now show that there exists a large enough radius $r>0$ such that 
\begin{align}
\label{induction_start}
 \fint_{B_r^+(x_0)} |\eps \nabla \theta^{\eps}_i |^2 \dx \leq \Big(\frac{\eps}{r} \Big)^{\frac{2}{3}}.
\end{align}

To show this, we introduce the event that \eqref{induction_start} fails for all $r>0$ and denote it by $\mathcal{B}$. We then have that 
\begin{align}
\mathbb{P} \big[  \mathcal{B} \big] \leq \mathbb{P} \Big[  \fint_{B_r^+(x_0)} |\eps\nabla \theta^{\eps}_i |^2 \dx  \geq   \Big(\frac{\eps}{r} \Big)^{\frac{2}{3}} \Big],  
\end{align}
for any fixed $r>0$. To estimate the right-hand probability, we notice that by the stationarity of $\nabla \theta^{\eps}_i$ tangentially to $\partial \Hd$ we have
\begin{align}
\en{\fint_{B_r^+(x_0)} |\eps \nabla \theta^{\eps}_i |^2 \dx } \lesssim \frac{1}{r} \en{\int_0^{\infty} |\eps \nabla \theta^{\eps}_i (x_0 + z)|^2 \, \dd z } \lesssim \frac{\eps}{r},
\end{align}
where the last inequality follows from \eqref{int_along_rays}. We then apply Markov's inequality to obtain
\begin{align}
\mathbb{P} \Big[   \fint_{B_r^+(x_0)} |\eps \nabla \theta^{\eps}_i |^2   \dx \geq  \Big(\frac{\eps}{r} \Big)^{\frac{2}{3}} \Big] \lesssim \Big(\frac{ \eps}{r} \Big)^{\frac{1}{3}},
\end{align}
for $r>1$.  Since this is true for any $r>0$, the Borel-Cantelli Lemma applied over dyadic scales implies that $\mathbb{P}[\mathcal{B}] =0$.\\

\medskip

\noindent {\bf Step 3: Comparison to constant coefficient problem in the straightened coordinates.} \qquad We now show that there exists a constant $\tau_0>0$ such that for any ratio $\tau\in (0,\tau_0]$ there exist constants $c(\domain,\tau)$, $C^*$, and a random field $\mathcal{C}(a,\cdot)$ with stretched exponential moments such that the estimate
\begin{align}
\label{Step_2_Lemma_4}
\fint_{B^+_r} |\eps \nabla \tilde{\theta}_i^{\eps}|^2 \dd y
\leq
\frac{1}{2} \tau^{-2/3} \fint_{B^+_{R}} |\eps \nabla \tilde{\theta}^{\eps}_i|^2 \dd y
+\C(a, x_0)\frac{\eps}{r}
+C^* \Big(\fint_{B^+_{R}} |\eps \nabla \tilde{\theta}_i^{\eps}|^2 \dd y\Big)^2
\end{align}
holds for any $R$ and $r:=\tau R$ with $c(\domain,\tau) \geq R \geq r \geq \C(a,x_0) \eps$.

\smallskip

We begin by setting up our argument: Notice that \wolog we may assume that $1 \leq r \leq R/ 16$, as otherwise \eqref{Step_2_Lemma_4} is trivial. Let $\zeta$ be a smooth cutoff for $\domain_{\eps}$ in $\domain$ --in particular, $\zeta \equiv1$ in a neighborhood of the boundary of width $\eps/2$ and $\zeta\equiv 0$ outside of $\domain_\eps$. Lastly, for some uniformly elliptic and bounded coefficient field $\bar b$ (not necessarily constant) to be determined, we introduce the solution $v^{\eps}$ to the equation
\begin{equation}
\label{v_general_1}
\begin{aligned}
-\nabla \cdot (\hat{\bar b} \nabla \widetilde{v^{\eps}}) &=  0&& \text{in } \quad  B_{R/2}^+,\\
\widetilde{v^{\eps}} & = \eps \widetilde{\theta^{\eps}_i} - \widetilde{\zeta}\, \widetilde{\phi_i^{\eps}} && \text{on } \quad \partial B_{R/2}^+.
\end{aligned}
\end{equation}  

Starting in our non-flattened coordinates we define the ``homogenization error'' 
\begin{align}
\label{DefinitionW}
w^{\eps}:= \eps \theta^{\eps}_i - \zeta \phi^{\eps}_i - (v^{\eps} + \eta(\phi^{\eps}_j - \eps \theta^{\eps}_j) \partial_j v^{\eps}),
\end{align}
where, for now, $\eta$ is an arbitrary smooth function such that $\eta = 0$ on $\partial B_{R/2}(x_0)$. Notice that $w^{\eps} \equiv 0$ on $\partial (\gamma(B_{R/2}^+))$ and satisfies the relation
\begin{align*}
-\nabla \cdot (a^{\eps}\nabla w^{\eps})
&=-\underbrace{\eps \nabla \cdot (a^{\eps}\nabla \theta^{\eps}_i)}_{=0}+\nabla\cdot (a^{\eps}\nabla (\phi^{\eps}_i \zeta))
+\nabla \cdot(a^{\eps}(\phi^{\eps}_j-\eps \theta^{\eps}_j) \nabla (\eta \partial_j v^{\eps}))
\\&~~~~
+\nabla \cdot (a^{\eps}  (e_j+\nabla \phi^{\eps}_j- \eps \nabla \theta^{\eps}_j) \eta \partial_j v^{\eps} ) + \nabla \cdot ((1-\eta) a^{\eps} \nabla v^{\eps})
\\& 
=\nabla\cdot (a^{\eps}\nabla (\phi^{\eps}_i \zeta))
+\nabla \cdot (a^{\eps} (\phi_j-\eps \theta^{\eps}_j) \nabla (\eta \partial_j v^{\eps}))
\\&~~~~
+\underbrace{a^{\eps} (e_j+\nabla \phi^{\eps}_j-\eps\nabla \theta_j^\eps) \cdot \nabla (\eta \partial_j v^{\eps})}_{=\nabla \cdot ((\bar{a}-\bar{b}) \eta \nabla v^{\eps})+ \nabla \cdot (\bar{b} \eta \nabla v^{\eps})
+ (\nabla\cdot \sigma^{\eps}_{j}-a^{\eps} \eps\nabla \theta^{\eps}_j)\cdot \nabla
(\eta\partial_j v^{\eps})}
+\nabla \cdot ((1-\eta) a^{\eps}\nabla v^{\eps}).
\end{align*}
Additionally using the skew-symmetry of $\sigma^{\eps}$, we find that $w^{\eps}$ solves 
\begin{equation}
\begin{aligned}
\label{cacc_intermediate}
-\nabla \cdot (a^{\eps} \nabla w^{\eps}) &= \nabla \cdot \Big(a^{\eps}\nabla (\zeta \phi^{\eps}_i) + (a^{\eps}(\phi^{\eps}_j-\eps\theta^{\eps}_j)-\sigma^{\eps}_j) \nabla (\eta \partial_j v^{\eps}) + (1-\eta)(a^{\eps}-\bar{b})\nabla v^{\eps} 
\\
&~~~\, \qquad \qquad  + (\bar{a} -\bar{b}) \eta \nabla v^{\eps} \Big) - a^{\eps} \eps \nabla \theta^{\eps}_j \cdot \nabla (\eta \partial_j v^{\eps})  &&\text{in }  \gamma(B_{R/2}^+) ,\\
w^{\eps} &= 0 &&\text{on }\partial (\gamma(B_{R/2}^+)),
\end{aligned}
\end{equation}
where we remark that the \rhs is not fully in divergence-form.  

Changing to our flattened coordinates, we find that $\tilde{w}$ solves 
\begin{equation}
\label{w_general_flattened_prop1}
\begin{aligned}
-\nabla \cdot (\widehat{a^{\eps}} \nabla \widetilde{w^{\eps}}) &= \nabla \cdot \Big(\widehat{a^{\eps}} \nabla (\widetilde{\phi^{\eps}_i \zeta}) + (1-\tilde{\eta}) (\widehat{a^{\eps}} -\widehat{\bar{b}}) \nabla \widetilde{v^{\eps}} +  \tilde{\eta} (\widehat{\bar a} -\widehat{\bar{b}}) \nabla \widetilde{v^{\eps}}) && \\
& \qquad  \qquad \qquad + (\widehat{a^{\eps}} (\widetilde{\phi_i^{\eps}} - \eps\widetilde{\theta_i^{\eps}}) - \widehat{\sigma_i^{\eps}})  \nabla(\tilde{\eta}  \partial_{\xi} \widetilde{v^{\eps}} ) \widetilde{\frac{\partial \gamma^{-1}_\xi}{ \partial x_{i}}}
\\
& \qquad  \qquad \qquad \qquad  \qquad + (a^{\eps,1} (\widetilde{\phi^{\eps}_i} - \eps \widetilde{\theta^{\eps}_i}) - \sigma^{\eps,1}_i )  \tilde{\eta}  \partial_{\xi} \widetilde{v^{\eps}} \stackon[-8pt]{$\nabla \frac{\partial \gamma_{\xi}^{-1}}{\partial x_i}$}{\vstretch{1.4}{\hstretch{2}{\widetilde{\phantom{\;\;\;\;\;\;\;\;}}}}}\Big) &&
\\
& \quad -  \widehat{a^{\eps}} \eps \nabla \widetilde{\theta^{\eps}_i} \cdot \nabla (\tilde{\eta}   \partial_{\xi} \widetilde{v^{\eps}}) \widetilde{\frac{\partial \gamma^{-1}_\xi}{ \partial x_{i}}} - a^{\eps,2} \eps \nabla \widetilde{\theta_i^{\eps}} \cdot \tilde{\eta}  \partial_{\xi} \widetilde{v^{\eps}} \stackon[-8pt]{$\nabla \frac{\partial \gamma_{\xi}^{-1}}{\partial x_i}$}{\vstretch{1.4}{\hstretch{2}{\widetilde{\phantom{\;\;\;\;\;\;\;\;}}}}} && \text{ in } B_{R/2}^+,
\\
\tilde{w} & = 0&&  \text{ on } \partial B_{R/2}^+,
\end{aligned}
\end{equation} 
where we emphasize that we use the Einstein summation convention in terms of $\xi = 1, \ldots, d$ and $i = 1, \ldots, d$, and, furthermore, the definition \eqref{flattened_coeff_general} along with 
\begin{align}
\label{flattened_coeff_general_second}
\begin{split}
b^1_{\mu \beta} (\cdot) := |\det(D \gamma)|  \stackon[-8pt]{$ \frac{\partial \gamma_{\mu}^{-1}}{\partial x_{\alpha}} b_{\alpha \beta}  $}{\vstretch{1.4}{\hstretch{2}{\widetilde{\phantom{\;\;\;\;\;\;\;\;}}}}}(\cdot)
\end{split}
\end{align}
applied for $b=a^{\eps}$ and $b=\sigma^{\eps}_i$, and 
\begin{align}
\label{flattened_coeff_general_third}
\begin{split}
a^2_{\alpha \nu} (\cdot) := |\det(D \gamma)|  \stackon[-8pt]{$  a_{\alpha \beta} \frac{\partial \gamma_{\nu}^{-1}}{\partial x_{\beta}} $}{\vstretch{1.4}{\hstretch{2}{\widetilde{\phantom{\;\;\;\;\;\;\;\;}}}}}(\cdot).
\end{split}
\end{align}
When $\domain$ is a half-space, then the last term under the divergence on the \rhs of \eqref{w_general_flattened_prop1} vanishes, as well as the last term on the right-hand side.

 \medskip
 
We now start the core of our argument: Thanks to the boundary data of $v^{\eps}$, notice that
\begin{align*}
\eps \widetilde{\theta_i^{\eps}} - \widetilde{\zeta \phi_i^{\eps}} - (\widetilde{v^{\eps}}(0) +\partial_j  \widetilde{v^{\eps}}(0) (\gamma_j + \widetilde{\phi^{\eps}_j} - \eps \widetilde{\theta_j^\eps})) = \eps\widetilde{\theta^{\eps}_i} - \widetilde{\zeta \phi^{\eps}_i} - \partial_1 \widetilde{v^{\eps}}(0) (\gamma_1 + \widetilde\phi^{\eps}_1 - \eps\widetilde\theta^{\eps}_1).
\end{align*}
Furthermore, $\eps \widetilde{\theta^{\eps}_i} - \widetilde{\zeta \phi^{\eps}_i} - \partial_1 \widetilde{v^{\eps}}(0) (y_1 + \widetilde\phi^{\eps}_1 - \eps\widetilde\theta^{\eps}_1)$ solves the equation 
\begin{equation}
\begin{aligned}
-\nabla \cdot (\widehat{a^{\eps}}\nabla (\eps\widetilde\theta^{\eps}_i - \widetilde{\zeta \phi^{\eps}_i} - \partial_1 \widetilde{v^{\eps}}(0) (y_1 + \widetilde\phi^{\eps}_1 - \eps\widetilde\theta^{\eps}_1))) &= \nabla \cdot  (\widehat {a^\eps} \nabla (\widetilde{\zeta \phi^\eps_i}) \\
& \qquad   + \partial_1 \widetilde{v^{\eps}}(0) \widehat{a^{\eps}}  \nabla  (y_1 - \gamma_1)) &&\text{in } B_{R/2}^+, \nonumber
\\
\eps \widetilde\theta^{\eps}_i - \widetilde{\zeta \phi^{\eps}_i} - \partial_1 \widetilde v^{\eps}(0) (y_1 +\widetilde \phi^{\eps}_1 - \eps\widetilde\theta^{\eps}_1) &= 0 &&\text{on } \partial B_{R/2}^+. \nonumber
\end{aligned}
\end{equation}
Emulating the classical proof for the Caccioppoli estimate, we find that 
\begin{align}
\label{Cacc_intermediate_1}
\begin{split}
& \int_{B_r^+} \big|\eps \nabla \widetilde \theta^{\eps}_i - \partial_1 \widetilde v^{\eps}(0) (e_1+\nabla \widetilde\phi^{\eps}_1-\eps\nabla \widetilde \theta^{\eps}_1) \big|^2 \dd y
\\
&  \lesssim  \int_{B_{2r}^+} |\nabla(\widetilde{\zeta \phi^\eps_i})|^2 + |\partial_1 \widetilde{v^\eps} (0)|^2 \ |e_1 - \nabla \gamma_1|^2 \dd y +  \frac{1}{r^2} \int_{B_{2r}^+} \big| \eps \widetilde{\theta^\eps_i} - \widetilde{\zeta \phi^\eps_i}  - \partial_1 \widetilde{v^\eps}(0) (y_1 + \widetilde{\phi^\eps_1} - \eps \widetilde{\theta^\eps_1}) \big|^2 \dd y,
\end{split}
\end{align}
where we have used that $B_{2r}^+ \subseteq B_{R/2}^+$. 

We now, for some $\rho\leq R/16$ to be determined, set $\eta$ such that $\tilde \eta$ is supported in $B_{R/2-\rho}$, $\tilde \eta = 1$ in $B_{R/2-2\rho}$, and $|\nabla \tilde{\eta}| \lesssim \rho^{-1}$. Towards showing \eqref{Step_2_Lemma_4}, we now use the definition of $w^{\eps}$ (see \eqref{DefinitionW}), the Poincar\'e inequality on $B_{2r}^+$ with homogeneous Dirichlet boundary data on $(\partial B_{2r}^+)\cap B_{2r}^+$, that $\tilde v(0) = 0$, and $r \leq R/16$ and $\delta \leq 1/2$ to obtain
\begin{align*}
& \frac{1}{r^2} \int_{B_{2r}^+} \big|\eps \widetilde{\theta^{\eps}_i} - \widetilde{\zeta \phi^{\eps}_i} - \partial_1 \widetilde{v^{\eps}}(0) (y_1 + \widetilde{\phi^{\eps}_1}- \eps \widetilde{\theta^{\eps}_1}) \big|^2 \dd y
\\
& \lesssim  \int_{B^+_{R /2}}  \Big( 1 - \frac{2|y|}{R}\Big)^{\kappa} |\nabla \widetilde{ w^{\eps}}|^2 \dd y
+ r^{d-2} \sup_{y\in B^+_{2r}} \big|\widetilde{v^{\eps}}(y) - (\widetilde{v^{\eps}}(0) + \partial_j \widetilde{v^{\eps}}(0) y_j) \big|^2
\\
& \,\,\,\,\,+ r^{-2} \sup_{ B^+_{2r}} \big|\nabla \widetilde{v^{\eps}} - \nabla \widetilde{v^{\eps}}(0)\big|^2 \int_{B_{2r}^+} |\widetilde{\phi^{\eps}}|^2 + |\eps \widetilde{\theta^{\eps}}|^2 \dd y,
\end{align*}
where we have also used that the weight in the second line satisfies $(1-2|y|/R)^{\kappa} \gtrsim 1$ in $B_{R/4}^+$. We remark that the (so far arbitrary) exponent $\kappa$ will be chosen so that Lemma~\ref{Hardy} will be applicable below. Continuing, we find that
\begin{align*}
& \frac{1}{r^2} \int_{B_{2r}^+} \big| \eps \widetilde{\theta^\eps_i} - \widetilde{\zeta \phi^{\eps}_i}  - \partial_1 \widetilde{ v^{\eps}}(0) (y_1 + \widetilde{\phi^{\eps}_1} - \eps \widetilde{\theta^\eps_1}) \big|^2 \dd y
\\
& \lesssim   \int_{B^+_{R/2}}  \Big(1 - \frac{2|y|}{R}\Big)^{\kappa} |\nabla \widetilde{w^{\eps}}|^2 \dd y
+ \sup_{B_{2r}^+} |\nabla^2 \widetilde{v^{\eps}}|^2  \Big( r^{d+2}+ \int_{B_{2r}^+}  |\widetilde{\phi^{\eps}}|^2 + |\eps \widetilde{\theta^{\eps}}|^2 \dd y \Big).
\end{align*}
Introduce the notation $\chi(\domain \neq \mathbb{H}^d)$ to denote the indicator of the case that $\domain$ is a bounded $C^{1,1}$-domain. We then apply \eqref{weighted_energy} from Lemma~\ref{Hardy} to $\tilde w$ in $B^+_{R/2}$ and use equation \eqref{w_general_flattened_prop1} to obtain
\begin{align}
\label{C_update_1}
\begin{split}
& \frac{1}{r^2} \int_{B_{2r}^+} \big|\eps \widetilde{\theta^\eps_i} - \widetilde{\zeta \phi^{\eps}_i}  - \partial_1 \widetilde{ v^\eps }(0) (y_1 + \widetilde{\phi^\eps_1} - \eps \widetilde{\theta^\eps_1}) \big|^2 \dd y
\\
& \lesssim  \int_{B_{R/2}^+} |\nabla (\widetilde{\zeta \phi^\eps_i})|^2 \dd y
\\&~~~
+ \sup_{B^+_{R/2- \rho}}  ((\rho^{-2} +  \chi(\domain \neq \mathbb{H}^d)) |\nabla \widetilde{v^{\eps}}|^2 + |\nabla^2 \widetilde{v^{\eps}}|^2)  \int_{B_{R/2}^+}|\widetilde{\phi^{\eps}}|^2+|\eps \widetilde{\theta^{\eps}}|^2+|\tilde\sigma^\eps|^2 + R^2 |\eps \nabla \widetilde{\theta^{\eps}}|^2 \dd y
\\&~~~
+  \sup_{B_{2r}^+} |\nabla^2 \widetilde{v^{\eps}}|^2  \Big( r^{d+2}+ \int_{B_{2r}^+}  |\widetilde{\phi^{\eps}}|^2 + |\eps \widetilde{\theta^{\eps}}|^2 \dd y  \Big)
\\&~~~
+ \Big(\frac{\rho}{R}\Big)^{\kappa} \int_{B^+_{R/2}} |\nabla \widetilde{v^\eps}|^2    \dd y+  \int_{B^+_{R/2}}  |\hat{\bar{a}} - \hat{\bar{b}}|^2 | \nabla \widetilde{v^\eps}|^2 \dd y,
\end{split}
\end{align}
where we have crucially used that $\gamma$ is $C^{1,1}$ and our choice of $\eta$.

We now set $\bar b$ such that $\hat{\bar b} = \hat{\bar a} (0)$. Furthermore plugging the bound \eqref{C_update_1} into \eqref{Cacc_intermediate_1} and using that $r\leq R/16$, we obtain
\begin{align}
\label{C_update_2_1}
\begin{split}
&\int_{B_r^+} \big| \eps \nabla \widetilde{\theta^\eps_i} - \partial_1 \widetilde{v^\eps}(0) (e_1+\nabla \widetilde{\phi^ \eps_1}- \eps \nabla \widetilde{\theta_1^\eps}) \big|^2 \dd y
\\&
\lesssim  \int_{B_{R/2}^+} |\nabla (\widetilde {\zeta \phi^\eps_i})|^2  \dd y 
\\&~~~
+   \sup_{B^+_{R/2 - \rho}}  ((\rho^{-2}+\chi(\domain \neq \mathbb{H}^d)) |\nabla \widetilde{v^\eps}|^2 + |\nabla^2 \widetilde{v^\eps}|^2)  \int_{B^+_{R/2}}|\widetilde{\phi^\eps}|^2+|\eps\widetilde{\theta^\eps}|^2+|\widetilde{\sigma^\eps}|^2 + R^2 |\eps\nabla  \widetilde{\theta^\eps}|^2  \dd y
\\&~~~
+  \sup_{B_{2r}^+} |\nabla^2 \widetilde{v^\eps}|^2  \Big(r^{d+2}  + \int_{B_{2r}^+} |\widetilde{\phi^\eps}|^2 + |\eps\widetilde{\theta^\eps}|^2  \dx \Big)  + \bigg(\Big(\frac{\rho}{R}\Big)^{\kappa} +\chi(\domain \neq \mathbb{H}^d) r^2\bigg) \int_{B_{R/2}^+} |\nabla \widetilde{v^\eps}|^2 \dd y
\\&~~~
+ \int_{B_{2r}^+} |\partial_1 \widetilde{v^\eps} (0)|^2 \ |e_1 - \nabla \gamma_1|^2 \dd y.
\end{split}
\end{align}
We now notice two things: First, combining the corrector bounds from the literature (see Theorem~\ref{CorrectorBoundGNO4}) with the Caccioppoli estimate applied to $\widetilde{\phi^\eps_i} + \gamma_i$, we obtain 
\begin{align}
\label{RegularityBoundNablaPhi}
\sup_{r\geq 1} \fint_{B_r} |e_i+\nabla \widetilde{\phi^\eps_i}|^2 \dd y
\leq \mathcal{C}(a,x_0).
\end{align}
We also use the constant coefficient regularity estimate \eqref{constant_reg_1} of Lemma \ref{constant_coeff_reg} applied to $\widetilde{v^\eps}$. In particular, notice that when $r \leq R/16$, the estimate 
\begin{align}
\label{RegularityConstantCoefficient_2}
\sup_{ B^+_{2r}}  |\nabla^2 \widetilde{v^\eps}|^2 \dd y\lesssim \frac{1}{R^2} \fint_{B^+_{R/2}} |\nabla \widetilde{v^\eps}|^2 \dd y
\end{align}
follows from \eqref{constant_reg_1}. Similarly, we have that 
\begin{align*}
|\partial_1 \widetilde{v^\eps} (0)|^2 \lesssim \fint_{B_{R/2}^+} | \nabla \widetilde{v^\eps} |^2 \dd y,
\end{align*}
whereby the regularity of $\gamma$ allows us to bound
\begin{align}
\label{RegularityConstantCoefficient_3}
 \int_{B_{2r}^+} |\partial_1 \widetilde{v^\eps} (0)|^2 \ |e_1 - \nabla \gamma_1|^2 \dd y \lesssim \chi(\domain \neq \mathbb{H}^d) r^{d+2} \fint_{B_{R/2}^+} | \nabla \widetilde{v^\eps} |^2 \dd y.
\end{align}
Notice that in \eqref{RegularityConstantCoefficient_3} we have used that $|e_1 - \nabla \gamma_1|^2 \leq cr^2$. 

Using the whole-space corrector bounds from Theorem~\ref{CorrectorBoundGNO4}; injecting \eqref{RegularityBoundNablaPhi}, \eqref{RegularityConstantCoefficient_2}, and \eqref{RegularityConstantCoefficient_3} into \eqref{C_update_2_1}; and also applying Lemma~\ref{constant_coeff_reg}, while also taking averages on both sides of the inequality and choosing $0<\rho\ll r$, we obtain
\begin{align}
\nonumber
& \fint_{B_r^+} \big| \eps \nabla \widetilde{\theta^{\eps}_i} - \partial_1 \widetilde{v}(0) (e_1+\nabla \widetilde{\phi^\eps_1}- \eps \nabla \widetilde{\theta^\eps_1})\big|^2 \dd y
\\&
\nonumber
\lesssim  \Big(\frac{R}{r} \Big)^d  \fint_{B_{R/2}^+} |\nabla ( \widetilde{ \zeta \phi^\eps_i})|^2 
\dd y\\ \nonumber
&~~~ + \Big(\frac{R}{r} \Big)^d ( \rho^{-2} + \chi(\domain \neq \mathbb{H}^d)) \Big(\frac{R}{\rho} \Big)^d \fint_{B_{R/2}^+} |\nabla \widetilde{v^\eps}|^2 \dd y \Big(\C(a,x_0)\eps^2 + \fint_{B_{R/2}^+} |\eps\widetilde{\theta^{\eps}}|^2 + R^2 |\eps \nabla \widetilde{\theta^\eps}|^2 \dd y \Big) \\ \nonumber
&~~~ +  R^{-2} \fint_{B_{R/2}^+} |\nabla \widetilde{v^\eps}|^2 \dd y \Big( r^2+ \C(a,x_0) \eps^2 + \fint_{B_{2r}^+} |\eps \widetilde{\theta^\eps}|^2 \dd y \Big) \\ \nonumber
&~~~ +  \Big( \frac{R}{r}\Big)^d \Big( \Big(\frac{\rho}{R}\Big)^{\kappa} + r^2 \chi(\domain \neq \mathbb{H}^d) \Big)\fint_{B_{R/2}^+} |\nabla \widetilde{v^\eps}|^2 \dd y \\ \nonumber
&
\lesssim  \Big(\frac{R}{r} \Big)^d  \fint_{B_{R/2}^+} |\nabla (\widetilde{\zeta \phi^\eps_i})|^2 
\dd y \,  \\ 
\label{C_update_2}
& ~~~
+     \fint_{B_{R/2}^+} |\nabla \widetilde{v^\eps}|^2 \dd y \bigg[  \Big(\frac{R}{r} \Big)^d \Big( \frac{R}{\rho}\Big)^{d+2}   (R^{-2} + \chi(\domain \neq \mathbb{H}^d))  \Big( \C(a, x_0) \eps^2  +  \fint_{B_{R/2}^+}| \eps\widetilde{\theta^{\eps}}|^2 + R^2 |\eps\nabla \widetilde{\theta^\eps}|^2 \dd y \Big)
\\&~~~~~~~~~~~~~~~~~~~~~~~~~~~~~~~~~~~~~~~~~~~~
+ R^{-2} \fint_{B_{2r}^+}|\eps\widetilde{\theta^\eps}|^2 \dd y +   \Big( \frac{R}{r}\Big)^d \Big(\Big(\frac{\rho}{R}\Big)^{\kappa} +r^2 \chi(\domain \neq \mathbb{H}^d) \Big) +\Big(\frac{r}{R}\Big)^2 \bigg].
\nonumber
\end{align}
To continue processing \eqref{C_update_2} we again collect some estimates: First, we notice that
\begin{align}
\label{BoundXiPhi}
\fint_{B^+_{R/2}} |\nabla (\widetilde{\zeta \phi^\eps_i})|^2 \dd y
\lesssim \mathcal{C}(a,x_0) \frac{\eps}{R},
\end{align}
which follows from the whole-space corrector bounds (see Theorem~\ref{CorrectorBoundGNO4}), the previously obtained \eqref{RegularityBoundNablaPhi}, and the definition of the cutoff $\zeta$.  Combining \eqref{BoundXiPhi} with the Poincar\'{e} inequality on $B_{R/2}^+$ with homogeneous Dirichlet boundary data on $\partial B_{R/2}^+ \cap B_{R/2}^+$ yields that
\begin{align}
\label{BoundTheta}
\begin{split}
\fint_{B^+_{R/2}} |\eps \widetilde{\theta^\eps_i}|^2 \dd y
&\lesssim
 \fint_{B^+_{R/2}} |\eps \widetilde{\theta^\eps_i}-\widetilde{\zeta \phi^\eps_i}|^2 \dd y
+ \fint_{B^+_{R/2}} |\widetilde{\zeta \phi^\eps_i}|^2 \dd y
\\&
\lesssim
R^2 \fint_{B^+_{R/2}} |\nabla (\eps \widetilde{\theta^\eps_i}-\widetilde{\zeta \phi^\eps_i})|^2 \dd y
+\mathcal{C}(a,x_0) \eps R
\\&
\lesssim R^2
\fint_{B^+_{R/2}} | \eps \nabla \widetilde{\theta_i^\eps}|^2 \dd y
+\mathcal{C}(a,x_0) \eps R.
\end{split}
\end{align}
The bound \eqref{BoundTheta} is clearly also valid for $R/2$ replaced by $2r$.

We also notice that by \eqref{v_general_1} and by our definition $\hat{\bar b}(y) := \hat{\bar a} (0)$ the function $\widetilde{v^\eps} - (\eps \widetilde{\theta_i^\eps} - \widetilde{\zeta \phi^\eps_i})$ solves the equation 
\begin{equation}
\label{equ_difference}
\begin{aligned}
-\nabla \cdot (\hat{\bar{a}}(0)\nabla ( \widetilde{v^\eps} - (\eps \widetilde{\theta_i^\eps} - \widetilde{\zeta \phi^\eps_i}))) &= \nabla \cdot (\hat{\bar{a}}(0) \nabla (\eps \widetilde{\theta_i^\eps} - \widetilde{\zeta \phi^\eps_i}))
&&\text{in }B_{R/2}^+,
\\
\widetilde{v^\eps} - (\eps \widetilde{\theta^\eps_i} - \widetilde{\zeta \phi^\eps_i}) &= 0&&\text{on }\partial B_{R/2}^+,
\end{aligned}
\end{equation}
which by the energy estimate yields 
\begin{align}
\label{energy_1}
\begin{split}
\fint_{B_{R/2}^+} | \nabla \widetilde{v^\eps} |^2 \dd y & \lesssim  \fint_{B_{R/2}^+} | \eps\nabla\widetilde{\theta^\eps_i}|^2  + |\nabla(\widetilde{\zeta \phi^\eps_i})|^2 \dd y
\\
& \stackrel{\eqref{BoundXiPhi}}{\lesssim} \fint_{B_{R/2}^+} |\eps\nabla \widetilde{\theta^\eps_i}|^2 \dd y + \mathcal{C}(a, x_0) \frac{\eps}{R}.
\end{split}
\end{align}
Inserting bounds \eqref{BoundXiPhi}, \eqref{BoundTheta}, and \eqref{energy_1} into \eqref{C_update_2} and additionally using 
\begin{align}
\label{PointwiseGradBound}
\begin{split}
& \fint_{B_r^+} |\partial_1 \widetilde{v^\eps} (0)|^2 |e_1+\nabla \tilde{\phi^\eps_1}- \eps\nabla \widetilde{\theta^\eps_1}|^2 \dd y \\
&\lesssim 
r^{-2} \fint_{B_{2R}^+} | \nabla \widetilde{v^\eps} |^2 \dd y \Big( \fint_{B_{2r}^+} | y_1 + \widetilde{\phi^\eps_1} - \eps \widetilde{\theta^\eps_1}|^2 \dd y + r^2 \Big),
\end{split}
\end{align}
which follows from the Caccioppoli inequality applied to $\gamma_1 + \widetilde{\phi^\eps_1} - \eps \widetilde{\theta^\eps_1}$, \eqref{constant_reg_1}, and the regularity of $\gamma$, 
yields
\begin{align}
\nonumber
& \fint_{B_r^+} \big|\eps \nabla \widetilde{ \theta_i^\eps }\big|^2 \dd y
\\&
\nonumber
\lesssim  \Big(\frac{R}{r} \Big)^d \C(a,x_0) \frac{\eps}{R}
\\ \nonumber
& ~~~
+     \Big( \fint_{B_R^+} |\eps \nabla \widetilde{\theta^\eps}|^2 \dd y + \C(a,x_0) \frac{\eps}{R} \Big)\\ 
\label{final_C_update_1}
& \hspace{1cm}\times \bigg[  \Big(\frac{R}{r} \Big)^d \Big( \frac{R}{\rho}\Big)^{d+2}  (R^{-2}+\chi(\domain \neq \mathbb{H}^d)) \Big( \C(a, x_0) \eps^2  +  \C(a,x_0)\eps R +R^2 \fint_{B_R^+(x_0)} |\eps \nabla \widetilde{\theta^\eps}|^2 \dd y \Big)
\\
\nonumber
& \hspace{2.0cm}+ R^{-2} \Big( \C(a,x_0)\eps r + r^2\fint_{B_{2r}^+}|\eps\nabla \widetilde{\theta^\eps}|^2 \dd y \Big)+  \Big( \frac{R}{r}\Big)^d \Big(\Big(\frac{\rho}{R}\Big)^{\kappa} + r^2 \chi(\domain \neq \mathbb{H}^d) \Big) +\Big(\frac{r}{R}\Big)^2\bigg]
\\ \nonumber
&~~~ + r^{-2} \Big(\fint_{B_{R}^+} |\eps \nabla \widetilde{\theta^\eps}|^2 \dd y+ \C(a, x_0) \frac{\eps}{R}\Big) \Big( r^2 + \C(a,x_0) \eps^2 + r^2 \fint_{B_{2r}^+(x_0)} |\eps \nabla \widetilde{\theta^\eps}|^2\dd y +\C(a,x_0)\eps r \Big).
\end{align}
Requiring $\tau_0$ to be small enough, choosing $\rho/R$ small enough depending on $r/R$, and then choosing the constants $c(\domain)$ and $\mathcal{C}(a,x_0)$ in the condition $c(\domain) \geq R \geq r \geq \C(a,x_0) \eps \geq \eps$ suitably, we obtain the desired estimate \eqref{Step_2_Lemma_4}.

\bigskip

\noindent {\bf Step 3: Iteration to smaller scales.} \qquad  We observe that the result \eqref{Step_2_Lemma_4} of the previous step implies that
\begin{align*}
\fint_{B^+_{\tau^{k+1} R_0}} |\eps \nabla \widetilde{\theta^\eps}|^2 \dx
\leq
\frac{1}{2} \tau^{-2/3} \fint_{B^+_{\tau^{k} R_0}} |\eps \nabla \widetilde{\theta^\eps}|^2 \dx
+\mathcal{C}(a,x_0) \frac{\eps}{\tau^{k+1} R_0}
+C^* \bigg(\fint_{B^+_{\tau^{k} R_0}} |\eps \nabla \widetilde{\theta^\eps}|^2 \dx \bigg)^2
\end{align*}
for any $R_0>0$ as long as $\tau^{k+1} R_0 \geq \C(a,x_0) \eps$. This enables us to inductively propagate an estimate of the form
\begin{align*}
\fint_{B^+_{\tau^{k} R_0}} |\eps\nabla \widetilde{\theta^\eps}|^2 \dx
\leq 4\mathcal{C}(a,x_0)\Big(\frac{\eps}{\tau^k R_0}\Big)^{2/3}
\end{align*}
as long as
\begin{align*}
C^* \bigg(4\mathcal{C}(a,x_0)\Big(\frac{\eps}{\tau^k R_0}\Big)^{2/3}\bigg)^2
\leq \mathcal{C}(a,x_0)\Big(\frac{\eps}{\tau^{k+1} R_0}\Big)^{2/3}.
\end{align*}
Note that the latter condition is verified as long as $\tau^{k+1} R_0 \geq (16 C^* \mathcal{C}(a,x_0))^{3/2} \eps$.
Returning to our original coordinate, the start of the induction is provided by \eqref{induction_start} in the case $\domain=\Hd$ respectively by \eqref{FirstBLCEstimateDomain} in the case of a bounded $C^{1,1}$ domain. This directly entails our desired result.
\end{proof}

\bigskip 

With the result of Lemma \ref{Prop_1_aux} in hand, the argument for Proposition \ref{intermediate_lemma} is now a simple matter of additionally using standard Schauder estimates at small scales.

\begin{proof}[Proof of Proposition \ref{intermediate_lemma}] Throughout our argument we make free use of the notation listed previously under ``Notation''. We set $\tilde{\C}(a, p(x_0))$ to be determined via Lemma \ref{Prop_1_aux}, and split our argument into two cases:\\

\smallskip

\noindent{\bf Case 1: $2 \dist(x_0, \partial \domain) \geq \tilde{\C}(a, p(x_0)) \eps$}  \qquad In this case the statement of Proposition~\ref{intermediate_lemma} is an immediate consequence of \eqref{lemma_3_eqn} applied around the point $p(x_0)$ with $r=2 \dist(x_0, \partial \domain)$ and the standard interior version of the large-scale regularity estimate \eqref{Reg06} applied with $r=\frac{1}{2}\dist(x_0, \partial \domain)$. We remark that the interior version of \eqref{Reg06} follows, \eg, via \cite[Theorem 1 \& Theorem 2]{GNO_final} combined with the argument for Corollary~\ref{cor_Reg06}.

\bigskip

\noindent{\bf Case 2: $2\dist(x_0, \partial \domain) \leq \tilde{\C}(a, p(x_0)) \eps$}  \qquad Let $\zeta$ be a cutoff with $\zeta=1$ on $\partial\domain$ and $\zeta = 0$ in $\domain_{\eps}^c$. Notice that $\eps\theta^{\eps}_i-\zeta \phi^{\eps}_i$ solves $-\nabla \cdot (a^{\eps}\nabla (\eps \theta^\eps_i-\zeta \phi^\eps_i)) = \nabla \cdot (a^\eps\nabla (\zeta \phi^\eps_i))$ in $\domain$ (with homogeneous Dirichlet boundary data). Thus, standard Schauder theory (see, \textit{e.g.}, \cite[Theorem 5.19]{GM_book}) yields 
\begin{align}
\label{Prop_1_2}
\begin{split}
&\sup_{B_{\eps}(x_0)\cap\domain} |\eps \nabla \theta^\eps_i|(x) \leq \| (\eps \nabla \theta^\eps_i)(\tfrac{\cdot-x_0}{\eps}) \|_{C^{0,\alpha}(B_1\cap \eps^{-1}(\domain-x_0))}\\
&\lesssim \C(a,x_0) \Bigg( \Big( \fint_{B_{2\eps}(x_0)\cap \eps^{-1}(\domain-x_0)} |\nabla (\eps \theta^\eps_i-\zeta \phi^{\eps}_i) |^2 \dx   \Big)^{\frac{1}{2}} +  \| (a^\eps \nabla(\eta \phi^\eps_i)) (\tfrac{\cdot-x_0}{\eps})  \|_{C^{0,\alpha} (B_{2}\cap \eps^{-1}(\domain-x_0))} \Bigg).
\end{split}
\end{align}
To treat the first term on the right-hand side of \eqref{Prop_1_2} we use the standard whole-space corrector bounds (Theorem~\ref{CorrectorBoundGNO4}) as well as the definition of $\zeta$ to the extent of 
\begin{align}
 \label{Prop_1_2.6}
& \Big( \fint_{B_{2\eps}(x_0)\cap \domain} |\nabla (\eps \theta_i^\eps-\zeta \phi^\eps_i) |^2 \dx   \Big)^{\frac{1}{2}}  \lesssim
 \Big( \fint_{B_{2\eps}(x_0)\cap \domain} |\eps \nabla \theta^\eps_i|^2 \dx   \Big)^{\frac{1}{2}} + \mathcal{C}(a,x_0).
\end{align}
Combining this with \eqref{lemma_3_eqn} of Lemma \ref{Prop_1_aux} yields
\begin{align*}
 \Big( \fint_{B_{2\eps}(x_0)\cap \domain} |\nabla (\eps \theta^\eps_i-\zeta \phi^\eps_i) |^2 \dx   \Big)^{\frac{1}{2}}  & \lesssim \C(a,x_0).
\end{align*}
To control the second term on the right-hand side of \eqref{Prop_1_2}, we make use of Lemma~\ref{CorrectorRegularity}, which holds thanks to the availability of assumption (A4), as well as the standard corrector bounds recalled in Theorem~\ref{CorrectorBoundGNO4}. In particular, we obtain
 \begin{align*}
\| (a^\eps \nabla(\eta \phi^\eps_i)) (\tfrac{\cdot-x_0}{\eps})  \|_{C^{0,\alpha} (B_{2}\cap \eps^{-1}(\domain-x_0))}
\lesssim \C(a,x_0) \| \nabla \phi^\eps_i  \|_{C^{0, \alpha} (B_{2\eps}(x_0))} + \eps^{-1} \| \phi^\eps_i  \|_{C^{0, \alpha}(B_{2\eps}(x_0))}  \lesssim \C(a,x_0).
 \end{align*}
This concludes the proof.
\end{proof}

\subsection{Proof of Lemma \ref{Hardy}: A weighted Hardy inequality}
\label{Prop_1_auxil}

To complete our argument for Proposition \ref{intermediate_lemma} it now only remains to give the proof of Lemma~\ref{Hardy}.
\begin{proof}[Proof of Lemma~\ref{Hardy}]
Letting $\eta(x) = \big(1-\frac{|x-x_0|}{r}\big)^{\frac{\kappa}{2}}$, we test the equation for $u$ with $\eta^2 u$. Using the uniform ellipticity of $a$, we obtain that 
\begin{align}
\label{aug_edit_1}
\begin{split}
&  \int_{B^+_r(x_0)} \Big(1-\frac{|x-x_0|}{r}\Big)^\kappa |\nabla u|^2 \dx \\
& \lesssim \frac{\kappa^2}{r^2} \int_{B^+_r(x_0)} \Big(1-\frac{|x-x_0|}{r}\Big)^{\kappa-2} u^2  \dx + \int_{B^+_r(x_0)} \Big(1-\frac{|x-x_0|}{r}\Big)^\kappa \big(|g|^2 + \frac{r^2}{\kappa^2} f^2\big)  \dx.
\end{split}
\end{align}
We use that $x_0 \in \partial \Hd$ to write the domain $B_r^+(x_0)$ in terms of $1$-dimensional filaments that originate at $x_0$ and terminate on the round part of $\partial B_r^+(x_0)$. On each of the rays we apply the $1$-dimensional Hardy's inequality (see, \textit{e.g.}, \cite[Theorem~9.16]{Zygmund})
\begin{equation}\nonumber
 \int_0^{r'} \rho^{\kappa-2} h^2(\rho) \ud \rho \lesssim \frac{1}{(1-\kappa)^2} \int_0^{r'} \rho^{\kappa} (h'(\rho))^2 \ud \rho,
\end{equation}
for $r'>0$, which holds under the assumption $h(0)=0$, to the effect of 
\begin{equation}\nonumber
 \int_{B^+_r (x_0) \setminus B^+_{r/3}(x_0)} \Big(1-\frac{|x-x_0|}{r}\Big)^{\kappa-2} u^2 \dx
\lesssim \frac{r^2}{(1-\kappa)^2} \int_{B^+_r (x_0) \setminus B^+_{r/3}(x_0)} \Big(1-\frac{|x-x_0|}{r}\Big)^{\kappa} |\nabla u|^2 \dx.
\end{equation}
Notice that we have removed the half-ball $B^+_{r/3}(x_0)$ from the domain of integration on the left-hand side. To finish the argument notice that for $x \in B^+_{2r/3}(x_0)$ the values of the weight $(1-|x-x_0|/r)^\kappa$ are uniformly bounded below away from $0$ and bounded above by $1$, which when combined with the Poincar\'e inequality yields
\begin{align*}
&\int_{B^+_{r/3}(x_0)} \Big(1-\frac{|x-x_0|}{r}\Big)^{\kappa-2} u^2 \dx
\\&
\lesssim
\int_{B^+_{2r/3}(x_0)\setminus B^+_{r/3}(x_0)} \Big(1-\frac{|x-x_0|}{r}\Big)^{\kappa-2} u^2 \dx
+r^2 \int_{B^+_{2r/3}(x_0)} \Big(1-\frac{|x-x_0|}{r}\Big)^{\kappa} |\nabla u|^2 \dx.
\end{align*}
We may then in particular choose $\kappa$ small enough to absorb the first term on the right-hand side of \eqref{aug_edit_1} into the left-hand side --thereby yielding \eqref{weighted_energy}.
\end{proof}

\section{Argument for Theorem \ref{large_scale_reg_domain}: Regularity for random $a$-harmonic functions on regular domains}
\label{operators_on_domain}

\subsection{Proof of Theorem \ref{large_scale_reg_domain}}
Using Proposition \ref{intermediate_lemma} we are now able to prove the large-scale regularity claimed in Theorem \ref{large_scale_reg_domain}. The proof is, in fact, quite similar to that of Proposition \ref{intermediate_lemma} --in particular, relying on a homogenization-infused Campanato iteration. 

We remark that the contents of Theorem \ref{large_scale_reg_domain} in the case that $\domain = \Hd$ have been shown in \cite{FischerRaithel}. In \cite{FischerRaithel} it is implicitly shown that the minimal radius in the half-space case is essentially a power of the whole-space minimal radius --thereby inheriting the uniform stretched exponential moments (shown for the whole-space case in \cite[Theorem 1 \& Theorem 2]{GNO_final}). Furthermore, while we restrict ourselves to a bounded $C^{1,1}$-domain below, this is a notational choice and the case $\domain = \Hd$ also follows from the below argument.

\smallskip 

\begin{proof}
\emph{Step 1:}
We first observe that it suffices to prove the result for $x_0 \in \partial \domain$. In particular, for $x_0\notin \partial\domain$ and for $R\leq 2\operatorname{dist}(x_0,\partial\domain)$, the result follows directly from the interior result of \cite{GNO_final} as well as \eqref{intermediate} (the latter being required for replacing the whole-space corrector $\phi^\eps$ by the corrector with Dirichlet boundary conditions $\phi^\eps-\eps \theta_i^\eps$). For points $x_0\notin \partial\domain$ with $R\geq 2\operatorname{dist}(x_0,\partial\domain)$), assume that the statements of the theorems are valid for $p(x_0)$ --recall that thanks to the regularity of $\partial \domain$ we can choose $\C(\domain)$ small enough so that $p(x_0)$, the closest point to $x_0$ on the boundary, is well-defined. We then deduce
\begin{align*}
\fint_{\domain \cap B_{\max\{2\operatorname{dist}(x_0,\partial\domain),2r\}}(p(x_0))} |\nabla u|^2 \dx &\lesssim \fint_{\domain \cap B_{R/2}(p(x_0))} |\nabla u|^2 \dx.
\end{align*}
For $r\geq \operatorname{dist}(x_0,\partial\domain)$ this directly entails the statement \eqref{large_scale_mvp_domain} of the theorem via an inclusion of balls, while for $r\leq \operatorname{dist}(x_0,\partial\domain)$ it does so in conjunction with the interior result of \cite{GNO_final}. The reduction of the proof of the excess-decay \eqref{excess_decay_thm} to the case of boundary points $x_0\in \partial\domain$ is achieved analogously.

\noindent
\emph{Step 2:}
Before comparing $u$ to the solution of the homogenized problem, we first flatten the boundary. We use the same notation as in Step 1 of the proof Lemma \ref{Prop_1_aux}. In particular, we assume that $c(\domain)$ is small enough so that $B_R(x_0) \cap \domain$ is in one chart, as otherwise \eqref{large_scale_mvp_domain} is trivial by simply choosing a large enough domain dependent constant. Furthermore, we may choose $c(\domain)$ small enough so that there exists $\delta: = \delta(\domain) \leq 1/2$ such that, for every $r \leq R$, we have $B_{(1-\delta)r} (x_0) \cap \domain \subseteq \gamma (B_r^+) \subseteq B_{(1+\delta)r}(x_0) \cap \domain$. 

Now, again emphasizing that we use the same notation as in Step 1 of the proof Lemma \ref{Prop_1_aux}, it suffices to show that $\tilde{u} := u(\gamma(\cdot))$ satisfies
\begin{align}
\label{mvp_thm_changed_coordinates}
\int_{B_r(x_0) \cap \domain} |\nabla \tilde{u}|^2 \, \dd y
\lesssim \Big(\frac{r}{R} \Big)^{d} \int_{B_R(x_0) \cap \domain} |\nabla \tilde{u}|^2 \, \dd y
\end{align}
as well as (after changing coordinates so that $\vec{n}(x_0)=-e_1$)
\begin{align}
\label{excess_decay_thm_changed_coordinates}
\begin{split}
& \inf_{A\in \mathbb{R}^d} \fint_{B_r^+} |\nabla \widetilde{u^\eps} - A \cdot \nabla ( \gamma_1 + \widetilde{\phi^\eps} - \eps \widetilde{\theta^\eps}) |^2 \dy\\
& \qquad  \lesssim_{d, \lambda, \domain} \Big(\frac{r}{R}\Big)^{2\alpha}
\inf_{A\in \mathbb{R}^d} \fint_{B_R^+} |\nabla \widetilde{u^\eps} - A \cdot \nabla ( \gamma_1 + \widetilde{\phi^\eps} - \eps \widetilde{\theta^\eps}) |^2 \dy,
\end{split}
\end{align}
for any $r,R$ with $c(\domain) \geq R \geq r \geq r_{\domain}^*(a^\eps,x_0)\geq \eps$.

In particular, assuming \eqref{mvp_thm_changed_coordinates}, we would have that
\begin{align}
\begin{split}
\int_{B_{(1-\delta)r}(x_0) \cap \domain} |\nabla u|^2 \, \dd x & = \int_{\gamma^{-1}(B_{(1-\delta)r}(x_0) \cap \domain)} \sum_{\beta=1}^d \Big| \frac{\partial \tilde{u}}{\partial y _{\nu}} \frac{\widetilde{\partial \gamma_{\nu}^{-1}}}{\partial x _{\beta}} \Big|^2  |\det (D \gamma)| \, \dd y\\
& \lesssim \int_{B_r(x_0) \cap \domain} |\nabla \tilde{u}|^2 \, \dd y\\
& \lesssim \Big(\frac{r}{R} \Big)^{d} \int_{B_R(x_0) \cap \domain} |\nabla \tilde{u}|^2 \, \dd y \lesssim  \Big(\frac{r}{R} \Big)^{d} \int_{B_{(1+\delta)R}(x_0) \cap \domain} |\nabla u |^2 \, \dd x,
\end{split}
\end{align} 
where the last inequality follows similarly to the first. Similarly, \eqref{excess_decay_thm} readily follows from \eqref{excess_decay_thm_changed_coordinates}.

\noindent
\emph{Step 3:}
Following the same strategy as in Step 3 of the proof Lemma \ref{Prop_1_aux}, we assume (without loss of generality) that $r \leq R/ 16$ and introduce the solution $v$ to the equation
\begin{equation}
\label{v_general}
\begin{aligned}
-\nabla \cdot (\hat{\bar b} \nabla \widetilde{v}) &=  0&& \text{in } \quad  B_{R/2}(x_0) \cap \domain,\\
\widetilde{v} & = \widetilde{u} && \text{on } \quad \partial (B_{R/2}(x_0) \cap \domain),
\end{aligned}
\end{equation} 
for the coefficient field $\bar b$ such that $\widehat{\bar b} = \widehat{\bar a}(0)$.  Defining the ``homogenization error'' now as 
\begin{align*}
w^\eps: = u -(v + \eta (\phi^\eps_i -\eps \theta^\eps_i) \partial_i v),
\end{align*}
where $\eta$ is set to be the same cut-off as in Step 3 of the proof Lemma \ref{Prop_1_aux}, we obtain by proceeding analogously to \eqref{cacc_intermediate}
\begin{equation}
\label{w_general}
\begin{aligned}
-\nabla \cdot (a^\eps \nabla w^\eps) &= \nabla \cdot ( (1-\eta) (a^\eps-\bar{a}) \nabla v + (\bar{a} -\bar b) \nabla v + ( a^\eps(\phi^\eps_i - \eps\theta^\eps_i) - \sigma^\eps_i ) \nabla(\eta \partial_i v)) && \\
& \qquad  - a^\eps \eps \nabla   \theta^\eps_i \cdot \nabla (\eta \partial_i v)&& \text{in }  \gamma(B_{R/2}^+),
\\
u & = 0 && \text{on } \partial (\gamma(B_{R/2}^+)).
\end{aligned}
\end{equation}

Changing coordinates and using the same notation as in Step 3 of the proof Lemma \ref{Prop_1_aux}, we find that $\widetilde{w^\eps}$ solves 
\begin{equation}
\label{w_general_flattened}
\begin{aligned}
-\nabla \cdot (\widehat{a^\eps} \nabla \widetilde{w^\eps}) &= \nabla \cdot \Big( (1- \tilde{\eta}) (\widehat{a^\eps}-\hat{\bar{a}}) \nabla \tilde{v} + (\widehat{\bar{a}} -\widehat{\bar{a}}(0)) \nabla \tilde{v} + (\widehat{a^\eps} (\widetilde{\phi_i^\eps} - \eps\widetilde{\theta_i^\eps}) - \widehat{\sigma_i^\eps})  \nabla(\tilde{\eta}  \partial_{\xi} \tilde{v} ) \widetilde{\frac{\partial \gamma^{-1}_\xi}{ \partial x_{i}}}  && \\
& \qquad \qquad + (a^{\eps,1} (\widetilde{\phi_i^\eps} - \eps \widetilde{\theta_i^\eps}) - \sigma^{\eps, 1}_i )  \tilde{\eta}  \partial_{\xi} \tilde{v} \stackon[-8pt]{$\nabla \frac{\partial \gamma_{\xi}^{-1}}{\partial x_i}$}{\vstretch{1.4}{\hstretch{2}{\widetilde{\phantom{\;\;\;\;\;\;\;\;}}}}}\Big) && \\
& \quad -  \widehat{a^\eps}  \eps \nabla \tilde{\theta^\eps_i} \cdot \nabla (\tilde{\eta}   \partial_{\xi} \tilde{v}) \widetilde{\frac{\partial \gamma^{-1}_\xi}{ \partial x_{i}}} - a^{\eps,2}\eps \nabla \widetilde{ \theta_i^\eps} \cdot \tilde{\eta}  \partial_{\xi} \tilde{v} \stackon[-8pt]{$\nabla \frac{\partial \gamma_{\xi}^{-1}}{\partial x_i}$}{\vstretch{1.4}{\hstretch{2}{\widetilde{\phantom{\;\;\;\;\;\;\;\;}}}}} && \hspace{-1.5cm}\text{in } B_{R/2}^+,
\\
\widetilde{w^\eps} & = 0&& \hspace{-1.5cm} \text{on } \partial B_{R/2}^+.
\end{aligned}
\end{equation}
Then notice that $\tilde{u} - \partial_1 \tilde{v}(0) (\gamma_1 + \widetilde{\phi_1^\eps} -\eps\widetilde{\theta_1^\eps})$ is $\widehat{a^\eps}$-harmonic in $B^+_{R(1+\delta)/2}$ and that  $\tilde{u} - \partial_1 \tilde{v}(0) (y_1 + \widetilde{\phi_1^\eps} -\eps\widetilde{\theta_1^\eps})$ has homogeneous Dirichlet boundary data on the flat part of $\partial B^+_{R/2}$. Since $\widehat{a^\eps}$ is uniformly elliptic and bounded, this implies a Caccioppoli estimate 
\begin{align}
\label{Cacc_straightened}
\begin{split}
& \int_{B_r^+} |\nabla \tilde{u} - \partial_1\tilde{v}(0) (\nabla \gamma_1 + \nabla \widetilde{\phi_1^\eps} - \nabla \widetilde{\theta_1^\eps}) |^2 \, \dd y \\
& \lesssim \frac{1}{r^2} \int_{B_{2r}^+} | \tilde{u} - \partial_1 \tilde{v}(0) (\gamma_1 + \widetilde{\phi_1^\eps} + \widetilde{\theta_1^\eps}) |^2 \, \dd y + \int_{B_{2r}^+} | \partial_1\tilde{v}(0) |^2 \, |e_1 - \nabla \gamma_1|^2  \, \dd y.
\end{split}
\end{align}
Combining \eqref{Cacc_straightened} with the same steps as used to obtain \eqref{C_update_1} in Step 3 of the proof of Lemma \ref{Prop_1_aux}, in particular applying Lemma \ref{Hardy} to $\widetilde{w^\eps}$ solving \eqref{w_general_flattened}, we find that 
\begin{align}
\label{C_update_1_flattened}
\begin{split}
&
\frac{1}{r^2} \int_{B_{2r}^+} | \tilde{u} - \partial_1 \tilde{v}(0) (\gamma_1 + \widetilde{\phi_1^\eps} + \widetilde{\theta_1^\eps}) |^2 \, \dd y
\\
& \lesssim \sup_{B^+_{R/2 - \rho}}  ((\rho^{-2}+1) |\nabla \tilde{v}|^2 + |\nabla^2 \tilde{v}|^2)  \int_{B_{R/2}^+}|\widetilde{\phi^\eps}|^2+|\eps \widetilde{\theta^\eps}|^2+|\widetilde{\sigma^\eps}|^2 + R^2 |\eps \nabla \widetilde{\theta^\eps}|^2 \, \dd y
\\&~~~
+ \sup_{B_{2r}^+} |\nabla^2 \tilde{v}|^2  \Big( r^{d+2}+ \int_{B_{2r}^+} |\widetilde{\phi^\eps}|^2 + |\eps\widetilde{\theta^\eps}|^2  \, \dd y \Big)  + \Big( \Big(\frac{\rho}{R}\Big)^{\kappa} +r^2 \Big)\int_{B_{R/2}^+} |\nabla \tilde{v}|^2 \dd y
\\
& ~~~ +  \int_{B_{R/2}^+} |(\widehat{\bar{a}} - \widehat{\bar{a}}(0)) \nabla \tilde{v}|^2 \dd y.
\end{split}
\end{align}
Using Lemma \ref{constant_coeff_reg} along with the standard whole-space corrector estimates as well as Lemma \ref{Prop_1_aux}, for $r\geq \eps$ we find that
\begin{align}
\label{new_term_1}
&\bigg(\sup_{B^+_{R/2 - \rho}} ((\rho^{-2}+1) |\nabla \tilde v|^2 + |\nabla^2 \tilde{v}|^2) \bigg) \int_{B_{R/2}^+} |\widetilde{\phi^\eps}|^2+|\eps\widetilde{\theta^\eps}|^2+|\widetilde{\sigma^\eps}|^2 + R^2 |\eps \nabla \tilde{\theta}^\eps|^2 \, \dd y
\\&\nonumber
\lesssim \C(a,x_0) R^d \left(\frac{R}{\rho}\right)^{d+2} \Big(\frac{\varepsilon}{r}\Big)^{\frac23} \fint_{B^+_{R/2}} |\nabla \tilde v|^2 \dd y.
\end{align}
Notice that in \eqref{new_term_1} we have also used that $R \leq c(\domain)$. Furthermore, using the regularity of $\gamma$ we have that
\begin{align}
\label{new_term_2}
\int_{B_{R/2}^+} |(\tilde{\bar{a}} - \tilde{\bar{a}}(0)) \nabla \tilde{v}|^2 \dd y \lesssim R^{d+2} \fint_{B_{R/2}^+} |\nabla \tilde{v}|^2 \dd y
\end{align}
and 
\begin{align}
\label{new_term_3}
 \int_{B_{2r}^+} | \partial_1\tilde{v}(0) |^2 \, |e_1 - \nabla \gamma_1|^2  \, \dd y \lesssim r^{d+2} \fint_{B^+_{R/2}} |\nabla \tilde{v}|^2 \dd y.
\end{align}
Following the arguments of Step 2 of the proof of Lemma \ref{Prop_1_aux} in combination with \eqref{new_term_1} and \eqref{new_term_2} yields
\begin{align}
\label{C_update_2_flattened}
\begin{split}
& \fint_{B_r^+} |\nabla \tilde{u} - \partial_1 \tilde{v}(0) (\nabla \gamma_1 + \nabla \tilde{\phi}_1 - \nabla \tilde{\theta}_1) |^2 \dd y
\\
& \lesssim  \fint_{B_{R/2}^+} |\nabla \tilde v|^2 \dd y \bigg[  \Big(\frac{R}{r} \Big)^d \Big( \frac{R}{\rho}\Big)^{d+2}  \C(a, x_0) \Big(\frac{\varepsilon}{r}\Big)^{\frac23}  + \Big(\frac{r}{R}\Big)^2 +\Big(\frac{R}{r} \Big)^d \Big(R^{2} +\Big(\frac{\rho}{R}\Big)^{\kappa}\Big) \bigg].
\end{split}
\end{align}
In particular,
first choosing the ratio $\tau:=r/R$ small enough and then choosing the ratio $\rho / R$ small enough, and then requiring $R$ to be sufficiently larger than $\varepsilon$ but much smaller than $1$, for any $\alpha\in (0,1)$ we can find some $\tau \in (0,1)$ such that the excess estimate
\begin{align}
\label{BasicIterationEstimate}
&\inf_{A\in \mathbb{R}^d} \fint_{B_{\tau R}^+} |\nabla \tilde{u} - A \cdot (\nabla \gamma_1 + \nabla \widetilde{\phi^\eps_1} - \eps\nabla \widetilde{\theta^\eps_1}) |^2 \dd y
\leq 
\tau^{2\alpha}
\fint_{B_R^+} |\nabla \tilde{u}|^2 \dd y
\end{align}
is valid for any $R\in [\tilde{\mathcal{C}}(a,x_0)\varepsilon,c(\domain)]$.

\noindent
\emph{Step 4:} In order to be able to iterate the estimate \eqref{BasicIterationEstimate}, first note that $\mathcal{Y}_1:=y_1+\widetilde{\phi_1^\eps}-\eps\widetilde{\theta_1^\eps}$ satisfies $-\nabla \cdot (\hat a \nabla \mathcal{Y}_1)=\nabla \cdot (\hat a (\nabla \gamma_1-e_1))$. Introducing $\mathcal{W}^R_1$ as the solution to $-\nabla \cdot (\hat a \nabla \mathcal{W}^R_1)=\nabla \cdot (\hat a (\nabla \gamma_1-e_1))$ in $B_R^+$ with homogeneous Dirichlet boundary conditions on $B_R^+$, we infer for any $\tilde A_R\in \mathbb{R}^d$
\begin{align*}
&\inf_{A\in \mathbb{R}^d} \fint_{B_{\tau R}^+} |\nabla \tilde{u} - (A+\tilde A_R) \cdot (\nabla \gamma_1 + \nabla \widetilde{\phi^\eps_1} - \eps\nabla \widetilde{\theta^\eps_1}) - \tilde A_R \cdot (\nabla y_1-\nabla \gamma_1-\nabla \mathcal{W}_1^R) |^2 \dd y
\\&
\leq 
\tau^{2\alpha}
\fint_{B_R^+} |\nabla \tilde{u}- \tilde A_R \cdot (\nabla y_1 + \nabla \widetilde{\phi^\eps_1} - \eps\nabla \widetilde{\theta^\eps_1}-\nabla \mathcal{W}_1^R)|^2 \dd y
\end{align*}
which entails
\begin{align*}
&\bigg(\inf_{A\in \mathbb{R}^d} \fint_{B_{\tau R}^+} |\nabla \tilde{u} - (A+\tilde A_R) \cdot (\nabla \gamma_1 + \nabla \widetilde{\phi^\eps_1} - \eps\nabla \widetilde{\theta^\eps_1})|^2 \dd y\bigg)^{1/2}
\\&
\leq 
\tau^{\alpha}
\bigg(\fint_{B_R^+} |\nabla \tilde{u}- \tilde A_R \cdot (\nabla \gamma_1 + \nabla \widetilde{\phi^\eps_1} - \eps\nabla \widetilde{\theta^\eps_1})|^2 \dd y\bigg)^{1/2}
+C(\tau) R |\tilde A_R|.
\end{align*}
Denoting by $A_R$ for any $R<1$ the minimizer of the excess
\begin{align*}
\operatorname{Exc}_R(x_0):=\bigg(\fint_{B_R^+} |\nabla \tilde{u}- \tilde A_R \cdot (\nabla \gamma_1 + \nabla \widetilde{\phi^\eps_1} - \eps\nabla \widetilde{\theta^\eps_1})|^2 \dd y\bigg)^{1/2}
\end{align*}
but setting $A_1:=0$ and $\operatorname{Exc}_1(x_0):=(\fint_{B_1^+} |\nabla \tilde u|^2 \dd y)^{1/2}$, the previous bound entails for $R\geq \mathcal{C}(a^\eps,x_0)$
\begin{subequations}
\label{IterationBound}
\begin{align}
|A_{\tau R}-A_R| &\leq C \operatorname{Exc}_{R}(x_0),
\\
\operatorname{Exc}_{\tau R}(x_0)
&\leq \tau^\alpha \operatorname{Exc}_{R}(x_0) + C R |A_R|.
\end{align}
\end{subequations}
We deduce $\operatorname{Exc}_{\tau^{k+1} R}(x_0)
\leq \tau^\alpha \operatorname{Exc}_{\tau^k R}(x_0) + C \tau^k R \sum_{\tilde k=1}^k \operatorname{Exc}_{\tau^{\tilde k} R}(x_0)$
and therefore
\begin{align*}
\sum_{k=1}^K \operatorname{Exc}_{\tau^k R}(x_0)
&\leq \tau^\alpha \sum_{k=1}^{K} \operatorname{Exc}_{\tau^{k-1} R}(x_0) + C R \sum_{k=1}^K \tau^k\sum_{\tilde k=1}^k \operatorname{Exc}_{\tau^{\tilde k-1} R}(x_0)
\\&
\leq \tau^\alpha \sum_{k=1}^{K} \operatorname{Exc}_{\tau^{k-1} R}(x_0) + C R \sum_{\tilde k=1}^K \tau^{\tilde k}  \operatorname{Exc}_{\tau^{\tilde k} R}(x_0),
\end{align*}
we obtain $\sum_{k=1}^K \operatorname{Exc}_{\tau^k R}(x_0)\leq C (\fint_{B_1^+} |\nabla \tilde u|^2 \dd y)^{1/2}$. This entails the bound $|A_R|\leq C (\fint_{B_1^+} |\nabla \tilde u|^2 \dd y)^{1/2}$ for all $R$ and thus the mean-value property.

Iterating this estimate back into \eqref{IterationBound} and iterating this bound, we obtain the excess-decay property \eqref{excess_decay_thm} for any $0<\alpha<1$ and any $r,R$ with $\tilde{\mathcal{C}}(a^\eps,x_0) \eps \leq r\leq R\leq c(\domain)$.
\end{proof}

\medskip

\subsection{Proof of Corollaries \ref{cor_Reg06} and \ref{GradientBoundByIntegral}: Pointwise estimates}

Adding the assumption (A4) to the conditions of Theorem \ref{large_scale_reg_domain}, the proofs of Corollaries \ref{cor_Reg06} and \ref{GradientBoundByIntegral} are now easy consequences of the combination of the large-scale regularity of Theorem \ref{large_scale_reg_domain} and the availability of Schauder estimates at small scales. Here are the arguments:

\begin{proof}[Proof of Corollary \ref{cor_Reg06}] Notice that the assumption of (A4) entails that classical Schauder theory is applicable on scales $r\leq \frac{\eps}{\mathcal{C}(a^\eps, x_0)}$. In particular, for $r\leq \frac{\eps}{\mathcal{C}(a^\eps,x_0)}$ we have that 
\begin{align*}
|\nabla u(x_0)| \leq \mathcal{C}(a^\eps, x_0) \Big( \fint_{B_r(x_0)\cap \domain} |\nabla u|^2 \dx \Big)^{\frac12}.
\end{align*}
Applying this estimate for $r=\frac{\eps}{\mathcal{C}(a^\eps, x_0)}$ and combining it with the bound \eqref{large_scale_mvp_domain} for $r=r^*_\domain(a^\eps,x_0)$, we arrive at the desired estimate \eqref{Reg06}.
\end{proof}

\bigskip 

\begin{proof}[Proof of Corollary \ref{GradientBoundByIntegral}]
Let $w$ be the solution to the dual problem
\begin{equation}
\label{equn_wT}
\begin{aligned}
-\nabla \cdot ((a^\eps)^\transpose\nabla w)&= \nabla \delta_{x_0} &&\text{in }\domain,
\\
w& =0&&\text{on }\partial\domain.
\end{aligned}
\end{equation}
We then have that
\begin{align}
\label{RewriteGradubyduality}
\begin{split}
\nabla u (x_0) & = \int_\domain (g \cdot \nabla w - f w ) \dx ,
\end{split}
\end{align}
and hence just need to obtain appropriate estimates for $w$.
For any $y \in \overline{\domain}$ with the notation $r:=\frac{1}{2}\dist(y,x_0)$, estimate \eqref{Reg06} provides the bound
\begin{align}
\label{Boundw}
|\nabla w(y)| \leq \mathcal{C}(a^\eps,y) \Big(\fint_{B_r(y) \cap \domain} |\nabla w|^2 \dx \Big)^{\frac 12}.
\end{align}
In order to estimate the \rhs we let $\eta$ be a continuous vector field supported in $B_r(y) \cap \domain$ and let $v \in \dot H^1_{0}(\domain)$ be a weak solution of
\begin{equation*}
\begin{aligned}
-\nabla \cdot (a^\eps \nabla v)  &= \nabla \cdot \eta&&\text{in }\domain,
\\
v&= 0&&\text{on }\partial\domain.
\end{aligned}
\end{equation*}
We then have
\begin{align}
\label{RewriteGradwbyduality}
\int_{\domain} \eta \cdot \nabla w \dx =
-\int_{\domain} a^\eps\nabla v \cdot \nabla w  \dx  =
-\int_{\domain} (a^\eps)^\transpose \nabla w \cdot \nabla v  \dx   = \nabla v (x_0).
\end{align}
Again applying estimate \eqref{Reg06}, now to $v$ which is $a^\eps$-harmonic in $B_{\frac{r}{2}}(x_0)$, and using the energy estimate for $v$, we deduce
\begin{align}
\label{Boundv}
|\nabla v(x_0)|
&
\leq   \mathcal{C}(a^\eps,x_0) \Big(\fint_{B_{\frac{r}{2}}(x_0) \cap \domain} |\nabla v|^2 \dx \Big)^{\frac 12}
\leq  \mathcal{C}(a^\eps,x_0) r^{-\frac{d}{2}}  \Big(\int_{\domain} |\eta|^2 \dx  \Big)^{\frac 12}.
\end{align}
Rewriting \eqref{Boundw} using the dual formulation of the $\dot H^1$-norm and inserting \eqref{RewriteGradwbyduality} and \eqref{Boundv}, we infer
\begin{align}
\label{BoundGradw}
|\nabla w(y)|
\leq \mathcal{C}(a^\eps,y) r^{-d/2} \sup_{\eta\, : \,  \supp \eta \subset B_r(y), \int |\eta|^2 \leq 1} \int_{\domain} \eta \cdot \nabla w \dx
\leq\mathcal{C}(a^\eps,y) \mathcal{C}(a^\eps,x_0) \frac{1}{|y-x_0|^d}.
\end{align}
Furthermore, since $w = 0$ on $\partial\domain$, we get that
\begin{align}
\label{BoundOnw}
|w(y)| \leq \mathcal{C}(a^\eps,y) \mathcal{C}(a^\eps,x_0) \frac{\dist(y,\partial\domain)}{|y-x_0|^d}.
\end{align}
Inserting \eqref{BoundGradw} and \eqref{BoundOnw} in \eqref{RewriteGradubyduality}, we infer \eqref{BoundSolutionRandomOperator}.
\end{proof}

\section{Argument for Theorem \ref{PropositionExcessDecayConvex}: Regularity for random $a$-harmonic functions on convex polytopes}
\label{regularity_near_corner}

Since the key step of our argument intends to be the transferring of regularity from the homogenized operator at large-scales, we must first show regularity for the homogenized operator. While these results on constant-coefficient regularity on convex cones and convex polytopes are not new, we provide their proofs below for the convenience of the reader. In particular, we begin by showing the following excess-decay:

\begin{proposition}
\label{PropositionExcessDecayConvex_constant} Let $\omega \subseteq \mathbb{S}^{d-1}$ be open with the property that $\overline{\cup_{\{s>0\}} s \omega}$ is strictly contained in a closed half-space. Let $x_0 \in \mathbb{R}^d$ and set $\domain:= \cup_{\{s>0\}} s \omega +x_0$.  

Then there exist constants $\delta:= \delta(\domain)>0$ and $C:=C(\domain)<\infty$ such that for any $R\geq r>0$ and any harmonic function $\bar u \in H^1(\domain\cap B_R(x_0))$ with homogeneous Dirichlet boundary data $\bar u=0$ on $\partial\domain \cap B_R(x_0)$ the estimate
\begin{align}
\label{excess_decay_cone_harmonic}
\fint_{B_r(x_0) \cap \domain} |\nabla \bar u|^2 \dx
\lesssim_{\domain} \Big(\frac{r}{R}\Big)^\delta
\fint_{B_R(x_0) \cap \domain} |\nabla \bar u|^2 \dx
\end{align}
holds.
\end{proposition}
\noindent As one sees in the proof of Proposition \ref{PropositionExcessDecayConvex_constant}, the cone being strictly contained in a half-space (a property that in particular excludes obtuse angles at the vertex of the cone) ensures that the corner contributions of $\bar u$ scale in such a way as to allow for \eqref{excess_decay_cone_harmonic}. Otherwise, for general cones, one would only expect \eqref{excess_decay_cone_harmonic} to hold for a certain regular contribution of $\bar u$ --see, e.\,g., \cite{Dauge_book}.

\medskip

Moving to the case of a convex polytope, we then obtain the following consequence of Proposition \ref{PropositionExcessDecayConvex_constant}.

\begin{corollary}
\label{CorollaryConstantCoefficientRegularityConvex} Let $\domain\subseteq \mathbb{R}^d$ be a convex polytope. Denote by $E$ the set of $(d-2)$-dimensional hyperedges of $\partial \domain$, and let $x_0\in \overline\domain$ and $0<\rho\leq \diam (\domain)$. 

Then there exist $\delta=\delta(\domain)$ such that any harmonic function $\bar u$ on $\domain\cap B_\rho(x_0)$ with homogeneous Dirichlet boundary conditions on $\partial\domain \cap B_\rho(x_0)$ satisfies 
\begin{align}
\label{BoundHarmonicFirstDerivative}
|\nabla \bar u(x_0)| \lesssim \Big(\frac{\dist(x_0,E)}{\rho}\Big)^\delta \Big(\fint_{\domain\cap B_{\rho}(x_0)} |\nabla \bar u|^2 \dx\Big)^{1/2}
\end{align}
and
\begin{align}
\label{BoundHarmonicSecondDerivative}
|\nabla^2 \bar u (x_0)| \lesssim \frac{1}{\rho^\delta \dist(x_0,E)^{1-\delta}} \Big(\fint_{\domain\cap B_{\rho}(x_0)} |\nabla \bar u|^2 \dx\Big)^{1/2}.
\end{align}
\end{corollary}
Given the regularity for a homogenized operator that is implied by Corollary \ref{CorollaryConstantCoefficientRegularityConvex}, in order to prove Theorem \ref{PropositionExcessDecayConvex} it only remains to argue that the regularity can be transferred at large scales. 

\subsection{Proof of Theorem \ref{PropositionExcessDecayConvex}}

\begin{proof}
We remark that it suffices to prove an $a^\eps$-harmonic analogue of Proposition~\ref{PropositionExcessDecayConvex_constant}, as then the proof of Corollary~\ref{CorollaryConstantCoefficientRegularityConvex} (in Section \ref{subsection_harmonic_polytopes}) applies \textit{mutatis mutandis} to the current situation.

We will, in particular, show the following: Let $\omega \subseteq \mathbb{S}^{d-1}$ and $x_0 \in \mathbb{R}^d$ be such that $\domain:=\cup_{\{s>0\}} s\omega + x_0$ is a convex cone satisfying the condition that $\overline{\domain}$ is strictly contained in a closed half-space. Assume additionally that $\domain\cap ((-1,1)^d+x_0)$ is a convex polytope. Then for any $r,R>0$ with $0<\eps\leq r^*(a^\eps, x_0) \leq r\leq R$ and any $a^\eps$-harmonic function $u \in H^1(\domain\cap B_R(x_0))$ with homogeneous Dirichlet boundary data $u=0$ on $\partial\domain \cap B_R(x_0)$, the estimate
\begin{align}
\label{ExcessDecayAHarmonicCone}
\fint_{B_r(x_0) \cap \domain} |\nabla u|^2 \dx
\lesssim \Big(\frac{r}{R}\Big)^\delta
\fint_{B_R(x_0) \cap \domain} |\nabla u|^2 \dx
\end{align}
holds. As always, it suffices to find constants $0<\tau_1 \leq \frac{1}{2}$, $0<\tau_2<1$ such that
\begin{align}
\label{MinimalRadiusNew}
\fint_{B_{\tau_1 R}(x_0) \cap \domain} |\nabla u|^2 \dx
\leq \tau_2
\fint_{B_R (x_0) \cap \domain} |\nabla u|^2 \dx
\end{align}
holds for any $r,R$ with $0<\eps\leq r^*(a^\eps, x_0) \leq r\leq R$ (as then an iteration yields \eqref{ExcessDecayAHarmonicCone}).

In order to establish \eqref{MinimalRadiusNew}, we proceed via a variation of the arguments in Lemma \ref{Prop_1_aux} and Theorem \ref{large_scale_reg_domain}. We first choose a cutoff $\eta$ with $\eta\equiv 0$ in $(B_R(x_0) \cap \domain)_{\rho}$, \ie in the $\rho$-neighborhood of $\partial (B_R(x_0) \cap \domain)$, and $\eta\equiv 1$ in $(B_R(x_0) \cap \domain) \setminus (B_R(x_0) \cap \domain)_{2\rho}$. Next, we choose a radius $ R^\prime \in [\tfrac{1}{2}R,\tfrac{3}{4}R]$ with $\fint_{\partial B_{R^\prime}(x_0)\cap \domain} |\nabla u|^2 \dS \lesssim  \fint_{B_R(x_0) \cap \domain} |\nabla u|^2 \dx$ and define $\bar u\in H^1(B_{R^\prime}(x_0) \cap \domain)$ to be the unique $\bar a$-harmonic function with $\bar u|_{\partial (B_{R^\prime}(x_0)\cap \domain)}=u|_{\partial (B_{ R^\prime}(x_0)\cap \domain)}$. For $z \in B_{R^\prime-\rho}(x_0) \cap \domain$, by Corollary~\ref{CorollaryConstantCoefficientRegularityConvex} and an energy estimate, the function $\bar u$ is subject to the bounds
\begin{subequations}
\label{BoundConstantCoefficientEquation}
\begin{align}
|\nabla \bar u(z)|^2 &\lesssim  \Big(\frac{R}{\rho}\Big)^d \Big(\frac{\dist(z,E)}{\rho}\Big)^{2\delta} \fint_{B_R(x_0)\cap \domain} |\nabla u|^2 \dx,
\\
|\nabla^2 \bar u(z)|^2 &\lesssim \frac{1}{\rho^{2\delta} \dist(z,E)^{2-2\delta}} \Big(\frac{R}{\rho}\Big)^d \fint_{B_R (x_0)\cap \domain} |\nabla u|^2 \dx.
\end{align}
Furthermore, by the Meyers estimate we have for some $p=p(\lambda,d)>1$
\begin{align}
\Big(\fint_{B_{R^\prime}(x_0)\cap \domain} |\nabla u|^{2p} \dx\Big)^{1/p} \lesssim \fint_{B_R(x_0) \cap \domain} |\nabla u|^{2} \dx.
\end{align}
\end{subequations}
Following the same computation as in Step 3 of the proof of Lemma \ref{Prop_1_proof}, we then find that the ``homogenization error'' $w^\eps = u-\bar u - \eta \phi^\eps_i \partial_i \bar u$ satisfies 
\begin{align*}
-\nabla \cdot (a\nabla w^\eps)
=\nabla \cdot ((a^\eps\phi^\eps_i-\sigma^\eps_i)\nabla (\eta \partial_i \bar u))+\nabla \cdot ((1-\eta)(a^\eps-\bar a)\nabla \bar u) \qquad \text{on } B_{R^\prime} (x_0) \cap \domain
\end{align*}
and is subject to homogeneous Dirichlet boundary data $w^\eps=0$ on $\partial (B_{ R^\prime}(x_0)\cap \domain)$. The energy estimate then yields 
\begin{align*}
&\int_{B_{R^\prime}(x_0)\cap\domain} |\nabla u-\nabla (\bar u+\eta \phi^\eps_i \partial_i \bar u)|^2 \dx
\\&
\lesssim \int_{B_{R^\prime}(x_0)\cap\domain} (|\phi^\eps|^2+|\sigma^\eps|^2) (\eta^2 |\nabla^2 \bar u|^2+|\nabla \eta|^2 |\nabla \bar u|^2) \dx + \int_{B_{ R^\prime}(x_0)\cap\domain} |1-\eta|^2 |\nabla \bar u|^2 \dx.
\end{align*}
Inserting the estimates \eqref{BoundConstantCoefficientEquation} and making use of the properties of the support of $\eta$, $1-\eta$, and $\nabla \eta$ as well as utilizing the corrector estimates from the literature (Theorem~\ref{CorrectorBoundGNO4}), we arrive at
\begin{align}
\label{ConeHomogenizationError}
\int_{B_{R^\prime}(x_0)\cap\domain} |\nabla u-\nabla (\bar u+\eta \phi^\eps_i \partial_i \bar u)|^2 \dx
& \lesssim \bigg( \Big(\frac{R}{\rho}\Big)^{d+2} \mathcal{C}(a^\eps,x_0) \frac{\eps^2}{R^2} + \Big(\frac{\rho}{R}\Big)^{(p-1)/p} \bigg) \int_{B_{R}(x_0)\cap\domain} |\nabla u|^2 \dx.
\end{align}
We also have by the triangle inequality
\begin{align*}
&\int_{B_{r}(x_0)\cap\domain} |\nabla (\bar u+\eta \phi_i^\eps \partial_i \bar u)|^2 \dx
\\&
\lesssim \sup_{B_r(x_0)\cap \domain } |\nabla \bar u|^2 \int_{B_{r}(x_0)\cap\domain}  \sum_i |e_i+\nabla \phi^\eps_i|^2 \dx \\
& \qquad +  \sup_{B_r(x_0)\cap \domain } |\nabla \bar u|^2 \int_{B_{r}(x_0)\cap\domain} |\nabla \eta|^2 |\phi^\eps_i|^2 \dx + \int_{B_r(x_0)\cap \domain } \eta^2 |\phi^\eps_i|^2 |\nabla^2 \bar u|^2 \dx
\end{align*}
which yields by \eqref{BoundConstantCoefficientEquation} and the corrector bounds (Theorem~\ref{CorrectorBoundGNO4})
\begin{align*}
&\fint_{B_{r}(x_0)\cap\domain} |\nabla (\bar u+\eta \phi^\eps_i \partial_i \bar u)|^2 \dx
\lesssim  \bigg( \Big(\frac{r}{R}\Big)^{2\delta} + \mathcal{C}(a^\eps,x_0) \frac{\eps^2}{\rho^2} +  \Big(\frac{R}{r}\Big)^{2-2\delta}  \mathcal{C}(a^\eps,x_0) \frac{\eps^2}{R^2} \bigg)\fint_{B_R(x_0)\cap \domain } |\nabla u|^2 \dx.
\end{align*}
Combining this estimate with the above bound for $\nabla u-\nabla (\bar u+\eta \phi_i \partial_i \bar u)$ in \eqref{ConeHomogenizationError}, we obtain
\begin{align*}
\fint_{B_{r}(x_0)\cap\domain} |\nabla u|^2 \dx
\lesssim
 \bigg(\Big(\frac{r}{R}\Big)^{2\delta} + \Big(\frac{R}{r}\Big)^d \Big(\frac{\rho}{R}\Big)^{(p-1)/p} + \mathcal{C}(a^\eps,x_0) \Big(\frac{R}{\min\{r,\rho\}}\Big)^{2+2d} \Big(\frac{\eps}{R} \Big)^2 \bigg)\fint_{B_R(x_0)\cap \domain } |\nabla u|^2 \dx.
\end{align*}
Choosing first $r/R$ small enough, then $\rho/R$ small enough, and then requiring $R\geq \eps \mathcal{C}(a^\eps,x_0)$, we arrive at our desired bound \eqref{MinimalRadiusNew}.
\end{proof}

\subsection{Proof of Proposition \ref{PropositionExcessDecayConvex_constant} and Corollary \ref{CorollaryConstantCoefficientRegularityConvex}: Harmonic regularity on convex polytopes} 
\label{subsection_harmonic_polytopes}

In this section we show that the convexity of a polytope is sufficient for implying $C^{1, \delta}$-regularity of harmonic functions (up to the boundary) for some $\delta>0$. 

\begin{proof}[Proof of Proposition \ref{PropositionExcessDecayConvex_constant}] We first remark that we may assume \wolog that $x_0=0$ and that $\domain \setminus \left\{0\right\} \subset \Hd$. In particular, rotation and translation leave the Laplacian invariant.

We now consider the $(d-1)$-dimensional manifold $\omega$ with boundary $\partial\domain\cap \mathbb{S}^{d-1}$ and its associated Laplace operator $-\Delta$ with homogeneous Dirichlet boundary conditions.
Denote by $(\lambda_i,\varphi_i)_{i\in \mathbb{N}}$ its eigenvalues and eigenfunctions; note that we may order these such that the $(\lambda_i)_i$ form an increasing sequence diverging to $+\infty$ and we may assume that the $\varphi_i$ form an orthonormal basis $\int_\omega \varphi_i \varphi_j \dS=\delta_{ij}$.

It is then a classical (and rather immediate, using the completeness of the basis $(\varphi_i)_i$ in $L^2(\domain)$) observation that any harmonic function $\bar u$ on $\domain\cap B_\rho$ with homogeneous Dirichlet boundary conditions on $\partial\domain \cap B_\rho$ may be expressed as
\begin{align}
\label{HarmonicFunctionSeries}
\bar u(x) = \sum_{i \in \mathbb{N}} a_i \bigg(\frac{|x|}{\rho}\bigg)^{b_i} \varphi_i\bigg(\frac{x}{|x|}\bigg)
\end{align}
for coefficients $a_i:=\int_{\omega} \bar u(\rho x) \varphi_i(x) \dS$ and with $b_i$ denoting the positive solution to $b_i(b_i-1)+(d-1)b_i=\lambda_i$ or equivalently $b_i=\tfrac{2-d}{2}+\sqrt{\lambda_i+\tfrac{1}{4}(d-2)^2}$.

This representation entails the estimate
\begin{align*}
\fint_{r \omega} |\bar u|^2 \dS
&\leq \sum_{i \in \mathbb{N}} \Big(\frac{r}{\rho}\Big)^{2b_i} \bigg| \rho^{-(d-1)} \int_{\rho \omega} \bar u(x) \varphi_i(x/\rho) \dS\bigg|^2
\leq \Big(\frac{r}{\rho}\Big)^{2b_1} \rho^{-(d-1)} \int_{\rho \omega} |\bar u|^2 \dS.
\end{align*}
Applying the Poincar\'e inequality on $\rho \omega$ to estimate the \rhs and integrating over all radii from $0$ to $2r$, we arrive at
\begin{align*}
\fint_{B_{2r} \cap \domain} |\bar u|^2 \dx
\lesssim \rho^2 \Big(\frac{r}{\rho}\Big)^{2b_1} \fint_{\rho \omega} |\nabla \bar u|^2 \dS.
\end{align*}
We next choose $\rho \in (R/2,R]$ such that
\begin{align*}
\fint_{\rho \omega} |\nabla \bar u|^2 \dS \lesssim \fint_{ B_R \cap \domain} |\nabla \bar u|^2 \dx
\end{align*}
and apply the Caccioppoli inequality on $B_{2r}\cap \domain$ to obtain
\begin{align*}
\fint_{B_{r} \cap \domain} |\nabla \bar u|^2 \dx
\lesssim \Big(\frac{r}{\rho}\Big)^{2b_1-2} \fint_{B_R \cap \domain} |\nabla \bar u|^2 \dx.
\end{align*}
By our assumption, $\overline{\domain}$ is strictly contained in a closed half-space. By eigenvalue comparison, we have $\lambda_1(\omega)>\lambda_1(\mathbb{H}_+^d\cap \mathbb{S}^{d-1})$. As $x_1|_{\mathbb{H}_+^d\cap \mathbb{S}^{d-1}}$ is the lowest eigenfunction of the Laplace operator on the half-sphere $\mathbb{H}_+^d \cap \mathbb{S}^{d-1}$, it follows that $b_1=b_1(\omega)>b_1(\mathbb{H}_+^d\cap \mathbb{S}^{d-1}(\domain))=1$ must hold for the domain $\domain$. This concludes our proof.
\end{proof}

\bigskip

\begin{proof}[Corollary \ref{CorollaryConstantCoefficientRegularityConvex}]
The proof of \eqref{BoundHarmonicFirstDerivative} proceeds as follows: For $1\leq k\leq d$, denote the set of $(d-k)$-dimensional hyperedges of $\partial\domain$ by $E_k$. For instance, in case $d=3$ denote by $E_1$ the set of faces of the polytope, denote the set of edges by $E_2$, and denote the set of vertices by $E_3$. Notice that $E =\cup_{k\in \{2,\ldots,d\}} E_k$.

For any point $x_0\in \domain$, set $r_{d+1}:=\diam( \domain)$ and $r_{k}:=\dist(x_0,E_k)$; note that in particular $r_1=\dist(x_0,\partial\domain)$ and that $r_k\leq r_{k+1}$. We can then find a sequence of points $x_d,\ldots,x_1$ with $x_k\in E_k$ as well as $|x_0-x_k|=r_k$.

For $k\geq 2$ and for radii $r,R$ with $0<r\leq R\leq c(\domain)$ we find that Proposition~\ref{PropositionExcessDecayConvex_constant} is applicable around the point $x_k$ (note that here we use the convexity of the polytope). Setting $R:=\min\{r_{k+1}/2,\rho/2\}$ and $r:=2r_{k}$ in Proposition~\ref{PropositionExcessDecayConvex_constant} yields (assuming that $2r_k\leq \rho/2$ and that $2r_k \leq \min\{r_{k+1},\rho\}$ -- if the latter condition is violated, the overall inequality however holds trivially)
\begin{align*}
\fint_{\domain\cap B_{r_k}(x_0)} |\nabla \bar u|^2 \dx
&\lesssim  \fint_{\domain \cap B_{2r_k}(x_k)} |\nabla \bar u|^2 \dx
\\&
\stackrel{\eqref{excess_decay_cone_harmonic}}{\lesssim}
\Big(\frac{2r_k}{\min\{r_{k+1}/2,\rho/2\}}\Big)^\delta
\fint_{\domain \cap B_{\min\{r_{k+1}/2,\rho/2\}}(x_k)} |\nabla \bar u|^2 \dx
\\&
\lesssim  \Big(\frac{r_{k+1}}{\min\{r_{k+1},\rho/2\}}\Big)^\delta \fint_{\domain\cap B_{\min\{r_{k+1},\rho\}}(x_0)} |\nabla \bar u|^2 \dx
\end{align*}
where the first and the third estimate are simply a consequence of an inclusion of balls.
For $k=1$ the similar bound
\begin{align*}
\fint_{\domain\cap B_{r_1}(x_0)} |\nabla \bar u|^2 \dx
&\leq C \fint_{\domain\cap B_{\min\{r_2,\rho\}}(x_1)} |\nabla \bar u|^2 \dx
\end{align*} 
holds using the regularity of harmonic functions on half-spaces instead of Proposition~\ref{PropositionExcessDecayConvex}. For $k=0$, we instead obtain by interior regularity of harmonic functions
\begin{align*}
|\nabla \bar u(x_0)|^2 \leq C \fint_{\domain\cap B_{\min\{r_1,\rho\}}(x_0)} |\nabla \bar u|^2 \dx.
\end{align*}
Combining these bounds yields \eqref{BoundHarmonicFirstDerivative}.

For $x_0 \notin E$, the bound \eqref{BoundHarmonicSecondDerivative} follows by combining \eqref{BoundHarmonicFirstDerivative} with the Caccioppoli inequality on $B_r(x_0)$ with $r:=c(\domain)\dist(x_0,E)$ to obtain
\begin{align*}
\fint_{\domain \cap B_{r/2}(x_0)} |\nabla^2 \bar u|^2 \dx
\leq C r^{-2} \fint_{\domain \cap B_{r}(x_0)} |\nabla \bar u|^2 \dx
\leq \frac{C}{\rho^{2\delta} \dist(x_0,E)^{2-2\delta}} \fint_{\domain\cap B_{\rho}(x_0)} |\nabla \bar u|^2 \dx.
\end{align*}
The pointwise estimate follows by $C^{2,\alpha}$-regularity theory on $B_{r/2}(x_0)\cap \mathbb{R}^d$ or (up to a rotation) on $B_{r/2}(x_0)\cap \mathbb{H}_+^d$.
\end{proof}

\section{Proof of Proposition \ref{sens_Lm0}: A weighted Meyers estimate}
\label{Meyers}

Our proof of the weighted Meyers estimate for the solution $v$ to \eqref{EquationV} relies on combining the standard Meyers estimate on members of a dyadic decomposition of $\domain$ with large-scale regularity results for the random elliptic operator. Since $v$ does not have Dirichlet boundary conditions on the various members of the decomposition, when we apply the classical Meyers estimate on each contribution, we pick up an additional term on the right-hand side (see, \textit{e.g.}, \eqref{Meyers_A}), the sum of which we handle by employing the large-scale regularity theory for $a$-harmonic functions on $\domain$. 

\begin{proof}
Recall that either $\domain = \Hd$ or $\domain$ is a bounded $C^{1,1}$-domain. Notice that for bounded domains $\domain$ we may assume that $R \ll c(\domain)$, since otherwise \eqref{Lemma_4_est_1} and \eqref{Lemma_4_est_2} hold trivially as a result of the Meyers estimate for \eqref{EquationV} and the definition \eqref{omega}. For our argument, for $f,g \in L^2(\domain)$, let $v_0, v_1 \in \dot H^1_{0} (\domain)$ be the (if $\domain = \Hd$, decaying) solutions of 
\begin{align}
\label{Lemma16_0}
- \nabla \cdot (a^\transpose(\tfrac{\cdot}{\eps}) \nabla v_0) = \nabla \cdot g  \text{ and }  - \nabla \cdot (a^\transpose(\tfrac{\cdot}{\eps}) \nabla v_1) = f \quad  \text{ both on } \domain, \, \,  \text{ with } v_0, v_1 = 0 \text{ on } \partial \domain. 
\end{align}
We, furthermore, introduce a domain dependent dyadic decomposition: Choose $\bar{r}$ such that $\diam( \domain) = 3 \times 2^{\tilde{k}} \times \bar{r}$ for some $\tilde{k} \in \mathbb{N}_0$ and $\bar{r} \sim R$, and let $J$ be the largest $j$ such that $\domain \cap B_{2^{j}\bar{r}} (x_0)$ is either a ball or is approximately the intersection of a half-space with a ball (\textit{i.e.}, $\domain \cap B_{2^{j}\bar{r}} (x_0)$ does not contain disjoint pieces of $\partial \domain$) for all $0 \leq j \leq J$. In the case that $\domain = \Hd$ we have that $J = \infty$, whereas $|J-\log_2 ( \diam(\domain) / \bar{r})|\lesssim 1$ when $\domain$ is a bounded $C^{1,1}$-domain. We then introduce
\begin{align*}
A_j :=
\begin{cases}
\left\{ |x-x_0| \leq 2 \bar{r} \right\}  \quad & \text{ for } j=0,\\
\left\{ 2^j \bar{r} < |x-x_0|  \leq 2^{j+1} \bar{r} \right\} \cap \domain \quad & \text{ for } 0<j<J,\\
\domain \setminus B_{2^j \bar{r}} (x_0) \quad & \text{ for } j = J.
\end{cases}
\end{align*}

The proof then proceeds in three steps: \\

\noindent{\bf Step 1:} \qquad  For $v_0$ solving \eqref{Lemma16_0}, the estimate
\begin{align}
\label{Lemma16_Step1}
\begin{split}
\Big( \int_{A_j} |\nabla v_0 |^2 \dx\Big)^{\frac{1}{2}} \lesssim \sum_{k=0}^{J}  \Big(2^{-|k-j|d} \int_{A_k} |g|^2 \dx \Big)^{\frac{1}{2}}
\end{split}
\end{align}
holds. As we will see, this follows via minor adaptions of \cite[Proof of Lemma 4.3]{BellaFehrmanFischerOtto}. We, in particular, make use of \eqref{large_scale_mvp_domain} of Theorem \ref{large_scale_reg_domain}.

To show the estimate \eqref{Lemma16_Step1}, we proceed by duality: Let $\rho \in L^2(\domain; \Rd^d)$ be compactly supported in $A_j$ and normalized in the sense that $\| \rho \|_{L^2(\domain)}  =1$, and let $w \in \dot H^1_0(\domain)$ solve 
\begin{equation}
\begin{aligned}
\label{Lemma16_w}
-\nabla \cdot (a(\tfrac{\cdot}{\eps}) \nabla w) & = -\nabla \cdot \rho &&\text{in } \domain,\\
w & = 0  && \text{on } \partial \domain.  
\end{aligned}
\end{equation}
Additionally using \eqref{Lemma16_0}, we have that 
\begin{align*}
\int_{\domain} \nabla v_0 \cdot \rho \dx = \int_{\domain} \nabla v_0 \cdot a(\tfrac{\cdot}{\eps}) \nabla w \dx = \int_{\domain} a^\transpose(\tfrac{\cdot}{\eps}) \nabla v_0 \cdot \nabla w \dx = - \int_{\domain} \nabla w \cdot g \dx,
\end{align*}
whereby it only remains to show that 
\begin{align}
\label{Lemma16_Step1_show} 
\left| \int_{A_k}  g \cdot \nabla w \dx \right| \lesssim \Big( 2^{-|k-j|d} \int_{A_k} |g|^2 \dx \Big)^{\frac{1}{2}}
\quad\quad\text{for all }k\in \{0,\ldots,J\}
\end{align}
(as then summing in $k$ and taking the supremum with respect to all possible $\rho$ yields \eqref{Lemma16_Step1}).

We first treat the case that $k \leq j+1$, for which showing that 
\begin{align}
\int_{A_k} |\nabla w|^2 \dx \lesssim (2^{k-j})^d
\end{align}
is enough. In the case that $j \leq k \leq j+1$ this follows directly from the energy estimate for \eqref{Lemma16_w}. For the case that $k < j $, we use that $w$ is $a^\eps$-harmonic in $B_{2^{j-1}\bar{r}} (x_0) \cap \domain$ to obtain
\begin{align*}
\int_{A_k} |\nabla w|^2 \dx \leq \int_{B_{2^{k+1} \bar{r}}(x_0) \cap \domain } |\nabla w|^2 \dx \stackrel{\eqref{large_scale_mvp_domain}}{\lesssim} \frac{|B_{2^{k+1} \bar{r}}(x_0) \cap \domain|}{|B_{2^{j-1} \bar{r}}(x_0) \cap \domain|} \int_{B_{2^{j-1} \bar{r}}(x_0) \cap \domain} |\nabla w|^2 \dx \lesssim 2^{d(k-j)}
.
\end{align*}
Notice that in the first and second inequalities above we have used that neither $k$ nor $j-1$ are equal to $J$ --in the second inequality we have also used the appropriate large-scale mean-value property. In the last inequality we have used the definition of $w$ and normalization of $\rho$.

We then consider the case that $k > j+1$, and let $u_k$ solve
\begin{equation}
\begin{aligned}
\label{Lemma16_wk}
\nabla \cdot (a^\transpose(\tfrac{\cdot}{\eps}) \nabla u_k) &= \nabla \cdot (\chi_{A_k} g ) &&  \text{ in } \domain,\\
u_k & = 0 && \text{ on } \partial \domain.
\end{aligned}
\end{equation}
In combination with \eqref{Lemma16_w}, \eqref{Lemma16_wk} implies that 
\begin{align*}
\int_{A_j} \nabla u_k \cdot \rho \dx = \int_{\domain} \nabla u_k \cdot \rho \dx = \int_{\domain} \nabla w \cdot a^\transpose(\tfrac{\cdot}{\eps}) \nabla u_k \dx = \int_{A_k} \nabla w \cdot g \dx.
\end{align*} 
Using the normalization of the vector-field $\rho$, we then obtain 
\begin{align*}
\left| \int_{A_k} \nabla w  \cdot g  \dx \right| \leq \Big( \int_{A_j} |\nabla u_k |^2  \dx \Big)^{\frac{1}{2}}.
\end{align*}
Since we have that $u_k$ is $a^\transpose$-harmonic on $B_{2^{k-1}\bar{r}}(x_0) \cap \domain $, we may apply the mean-value property \eqref{large_scale_mvp_domain} to obtain
\begin{align*}
\Big( \int_{A_j} |\nabla u_k|^2 \dx \Big)^{\frac{1}{2}}
&\leq \Big( \int_{B_{2^{j+1} \bar{r}}(x_0) \cap \domain} |\nabla u_k|^2 \dx \Big)^{\frac{1}{2}} 
\\&
\lesssim \Big( \frac{|B_{2^{j+1} \bar{r}}(x_0) \cap \domain|}{|B_{2^{k-1} \bar{r}}(x_0) \cap \domain|} \int_{B_{2^{k-1} \bar{r}}(x_0) \cap \domain} |\nabla u_k|^2  \dx\Big)^{\frac{1}{2}}  \lesssim  \Big( 2^{(j-k)d} \int_{A_k} |g|^2  \dx\Big)^{\frac{1}{2}} .
\end{align*} 
Notice that we have again used that $k\neq J$ and $j-1\neq J$.
\\

\noindent{\bf Step 2:} \qquad For $v_1$ solving \eqref{Lemma16_0}, the estimate
\begin{align}
\label{Lemma16_Step2}
\begin{split}
\Big( \int_{A_j} |\nabla v_1 |^2 \dx \Big)^{\frac{1}{2}} \lesssim \sum_{k=0}^{J} ((k-j)_+ +1)|A_0|^{\frac{1}{d}}\Big( 2^{-|k-j|d + 2 \max(k,j) } \int_{A_k} |f|^2 \dx \Big)^{\frac{1}{2}}
\end{split}
\end{align}
holds. This follows just like the corresponding step in \cite[Proof of Lemma 4.3]{BellaFehrmanFischerOtto}: We rewrite the PDE $-\nabla \cdot (a^\transpose \nabla u)=f$ as a PDE with a divergence-form right-hand side via solving $-\Delta V=f$ on $\domain$ with Neumann boundary conditions and setting $g:=-\nabla V$ to obtain $-\nabla \cdot (a^\transpose(\tfrac{\cdot}{\eps}) \nabla u)=\nabla \cdot g$. On a bounded domain, prior to this procedure we replace $f$ by $f-\smash{\frac{\fint_\domain f}{\fint_\domain \psi}\psi}$ for a smooth nonnegative function $\psi$ supported in $A_J$ (as otherwise the Neumann problem may have no solution) and estimate the solution to $-\nabla \cdot (a^\transpose(\tfrac{\cdot}{\eps}) \nabla u_{avg}) = \smash{\frac{\fint_\domain f}{\fint_\domain \psi}\psi}$ separately. The estimate \eqref{Lemma16_Step2} then follows from \eqref{Lemma16_Step1} and the decay properties of the Green's function of the Laplace operator and uses that $|A_j| \sim |A_0| 2^{jd}$ --which is also true for $A_J$, thanks to how we have defined $J$.
\\

\noindent{\bf Step 3:} \qquad We now conclude. Beginning with the argument for \eqref{Lemma_4_est_1}, we assume that $f =0$ and notice that the standard Meyers estimate applied on each $A_j$ gives
\begin{align}
\label{Meyers_A}
\Big( \fint_{A_j} |\nabla v |^{2p} \dx \Big)^{\frac{1}{2p}} \lesssim \Big( \fint_{A_j^+} |\nabla v|^2 \dx \Big)^{\frac{1}{2}} + \Big( \fint_{A_j^+} |g|^{2p} \dx \Big)^{\frac{1}{2p}},
\end{align}
where we use the convention
\begin{align*}
A_j^+ := \begin{cases}
A_0 \cup A_1 & \text{ if } j=0,\\
A_{j-1} \cup A_j \cup A_{j+1} & \text{ if } 0 < j < J,\\
A_{J-1} \cup A_J  & \text{ if } j=J. 
\end{cases}
\end{align*}
By the definition of $\omega_{\alpha_0,R}$ and our choice of $\bar{r}\sim R$, we know that $\omega_{\alpha_0,R} \lesssim 2^{\alpha_0 (j+1)}$ on $A_j$ and $2^{\alpha_0(j-1)} \lesssim \omega_{\alpha_0,R}$ on $A_j^+$  uniformly for $j \neq 0, J$. For $j =0$ we have that $1 \lesssim \omega_{\alpha_0,R}$ on $A_0$ and  $\omega_{\alpha_0,R} \lesssim 2^{\alpha_0}$ on $A_0^+$; in the case that $\domain$ is a bounded $C^{1,1}$-domain, for $j = J$ we have that $2^J \bar{r} \sim 1$ so that $R^{-\alpha_0}  \lesssim \omega_{\alpha_0,R} $ on $A_J$ and $ \omega_{\alpha_0,R} \lesssim R^{-\alpha_0}$ on $A_J^+$, where we emphasize that the constants depend on $\domain$. These observations allow us to smuggle-in the weight $\omega_{\alpha_0, R}$ and sum \eqref{Meyers_A} over $j$ to the extent of
\begin{align*}
\int_{\domain} |\nabla v|^{2p} \omega_{\alpha_0,R} \dx \lesssim \int_{\domain} |g|^{2p} \omega_{\alpha_0,R} \dx + \sum_{j = 0}^{J} (\max_{A_j} \omega_{\alpha_0,R}) |A_j|^{1-p} \Big( \int_{A_j} |\nabla v|^2 \dx \Big)^p.
\end{align*}
To complete the argument one now observes that by \eqref{Lemma16_Step1}, we have for any $\tau>0$ small enough
\begin{align}
\nonumber
&\sum_{j = 0}^{J} (\max_{A_j} \omega_{\alpha_0,R}) |A_j|^{1-p} \Big( \int_{A_j} |\nabla v|^2 \dx \Big)^p
\\&
\nonumber
\lesssim
\sum_{j = 0}^{J} 2^{\alpha_0 j} (2^j R)^{(1-p)d} \Big( \sum_{k=0}^{J} \Big( 2^{-|k-j|d } \int_{A_k} |g|^2 \dx \Big)^{1/2} \Big)^{2p}
\\&
\nonumber
\lesssim
\sum_{j = 0}^{J} 2^{\alpha_0 j} (2^j R)^{(1-p)d} \sum_{k=0}^{J} 2^{-|k-j|p(d-\tau)} (C 2^k R)^{d(p-1)} \int_{A_k} |g|^{2p} \dx
\\&
\nonumber
\lesssim
\sum_{k=0}^{J} 2^{\alpha_0 k} (2^k R)^{(1-p)d} (C 2^k R)^{d(p-1)} \int_{A_k} |g|^{2p} \dx
\\&
\label{Meyers_A_1}
\lesssim
\int_{\domain} |g|^{2p}  \omega_{\alpha_1,R} \dx.
\end{align}
Here, in the penultimate estimate the bound $pd>\alpha_0+(1-p)d$ or equivalently $\alpha_0<d(2p-1)$ entered.
 
We move on to the argument for \eqref{Lemma_4_est_2}. For this we assume that $g = 0$ and apply the version of Meyers' estimate for non-divergence form equations on each $A_j$. This gives
\begin{align*}
\Big( \fint_{A_j} |\nabla v |^{2p} \dx \Big)^{\frac{1}{2p}} \lesssim \Big( \fint_{A_j^+} |\nabla v|^2 \dx \Big)^{\frac{1}{2}} +  |A_j|^{1/d}\Big( \fint_{A_j^+} |f|^{2p} \dx \Big)^{\frac{1}{2p}},
\end{align*}
where we use that $|A_j|^{\frac{1}{d}} \sim |A_0|^{\frac{1}{d}} 2^j$ and $2^j \lesssim \frac{|x|}{R} +1$ on $A_j$ to further write
\begin{align*}
\int_{\domain} |\nabla  v |^{2p} \omega_{\alpha_0,R} \dx \lesssim |A_0|^{\frac{2p}{d}} \int_{\domain} |f|^{2p} \omega_{\alpha_0 + 2p,R} \dx + \sum_{j=0}^J \Big( \max_{A_j} \omega_{\alpha_0,R} \Big) |A_j|^{1-p} \Big( \int_{A_j} |\nabla v|^2 \dx \Big)^p.
\end{align*}
Showing that 
\begin{align}
\label{Meyers_B.1}
\sum_{j=0}^J \Big( \max_{A_j} \omega_{\alpha_0,R} \Big) |A_j|^{1-p} \Big( \int_{A_j} |\nabla v|^2 \dx \Big)^p \lesssim |A_0|^{\frac{2p}{d}} \int_{\domain} |f|^{2p}  \omega_{\alpha_1,R} \dx
\end{align}
finishes the argument. The relation \eqref{Meyers_B.1} follows from \eqref{Lemma16_Step2} similarly to the derivation of \eqref{Meyers_A_1} from \eqref{Lemma16_Step1}, but requiring now that $\alpha_0 < \alpha_1 - 2p$ and $\alpha_1 < d(2p-1)$.
\end{proof}

\appendix

\section{Exponentially localized boundary layer} 

\label{exp_sec}

As has become standard procedure for the whole-space corrector (see, \eg, \cite{GNO_final, FN_2020}), in our arguments in the case that $\domain = \Hd$ we make use of an exponentially localized version of $\theta$. In particular, for $T>0$ and $i = 1, \ldots, d$, we define the massive approximation for the boundary layer corrector $\theta_i^T \in H^1_{\loc}(\Hd)$ (with a localization parameter $T<\infty$, the boundary layer corrector $\theta_i$ being recovered in the limit $T\rightarrow \infty$) as the weak solution of 
\begin{equation}
\label{defn_theta_T}
\begin{aligned}
-\nabla \cdot (a\nabla \theta_i^T) + \frac{1}{T} \theta_i^T &=0 \quad\quad\text{in }\Hd,
\\
\theta_i^T &=\phi_i \quad\quad\text{on }\partial\Hd,
\end{aligned}
\end{equation}
where we remark that the Dirichlet boundary condition is complimented with a decay condition in the far-field. 

We remark that thanks to the massive term $\frac{1}{T}\theta_i^T$, \eqref{defn_theta_T} is easily solvable. In particular, for any $x_0 \in \partial \domain$ one first solves \eqref{defn_theta_T} with the boundary data $\theta_i^{T,R} = \phi^T_i \chi(B_R(x_0))$, where $B_R(x_0)$ is a $d-1$-dimensional ball in $\partial \Hd$. To pass to the limit $R \rightarrow \infty$ and to obtain the independence of the limit on the base point $x_0$, one can use following exponentially localized energy estimate:

\begin{lemma} \label{exp} Assume that the coefficient field $a$ satisfies the ellipticity and boundedness condition (A1). Let $T>0$ and let $L \geq \sqrt{T}$. Let $u \in H^1_{\loc}(\Hd)$ be a weak solution of 
\begin{equation}
\label{defn_u_T}
\begin{aligned}
-\nabla \cdot (a\nabla u ) + \frac{1}{T}u&= \nabla \cdot F && \text{in } \quad \, \Hd,\\
u &=  g &&\text{on }\quad \partial \Hd,
\end{aligned}
\end{equation}
where $g$ is the trace of $g \in H^1_{loc}(\Hd)$ and $F \in L^2_{\loc}(\Hd)$, such that $u$, $F$, and $g$ satisfy
\begin{align}
\label{polynomial_growth}
\limsup_{R \rightarrow \infty} R^{-k} \Big( \fint_{B_R^+} (|u| + |\nabla u| + |F| + |g| + |\nabla g|)^2 \dx\Big)^{\frac{1}{2}} =0 
\end{align}
for some $k \in \mathbb{N}_0$. Then there exists a constant $\gamma=\gamma(d,\lambda)\in (0,1]$ such that the estimate 
\begin{align}
\label{exp_weighted_energy_estimate} 
 \int_{\Hd} \Big(|\nabla u|^2 + \frac{1}{T}|u|^2\Big) \exp(-\gamma|x| /L) \, \dx \lesssim \int_{\Hd}\Big(|\nabla g|^2+ \frac{1}{T} |g|^2 +  |F|^2\Big)  \exp(-\gamma|x| /L) \, \dx
\end{align}
holds.
\end{lemma}

\begin{proof}
Letting $\eta = \exp(- \gamma |x|/L)$, we test \eqref{defn_u_T} with $(u-g)\eta$. We remark that testing with this function may be justified via approximation using \eqref{polynomial_growth}. After making use of $L \geq \sqrt{T}$ and Young's inequality, this yields 
\begin{align*}
& \int_{\Hd} \Big(|\nabla u|^2 + \frac{1}{T} |u|^2 \Big) \eta \dx \\
 &\lesssim \int_{\Hd} \Big( \eta a\nabla u \cdot \nabla g -(u-g)  \nabla \eta \cdot a \nabla u \dx + \frac{1}{T} \eta u g  - (u-g) \nabla \eta \cdot F - \eta \nabla(u-g) \cdot F \Big) \dx \\
& \lesssim  \int_{\Hd}  \eta \Big( \gamma  |\nabla u|^2 +\frac{\gamma}{T} |u|^2 + \frac{C(\gamma)}{T} |g|^2 + C(\gamma) |\nabla g|^2 + C(\gamma) |F|^2 \Big) \dx.
\end{align*}
Choosing a small enough $\gamma$ and absorbing the terms involving $u$ and $\nabla u$ then yields the claim. 
\end{proof}

 We remark that the estimate \eqref{exp_weighted_energy_estimate} yields the uniqueness of a sublinear solution $\theta_i^T$ to \eqref{defn_u_T}, which has the consequence that the stationarity of $\theta_i^T$ \wrt shifts parallel to $\partial \Hd$ follows from the stationary of $\phi_i$ and $a$.

\section{Regularity of random elliptic operators on $\Rd^d$ and corrector estimates}
\label{Appendix_A}

Here we first summarize the contents of the large-scale regularity results Theorems 1 and 2 of \cite{GNO_final}, which hold under the assumptions (A1)-(A3). While the companion paper to \cite{GNO_final} relies on an LSI assumption (as opposed to the spectral gap assumption in (A3)) to prove the required corrector estimates, the corrector bounds under the assumption (A3) may be taken from e.\,g.\ in \cite{FN_2020}.
\begin{theorem}[see \cite{GNO_final}] \label{summary_GNO} Let Assumptions (A1) -- (A3) be satisfied. Then there exists a random field $r^*=r^*(a,x)$ such that $\frac{r^*}{\varepsilon}$ has stretched exponential moments in the sense of \eqref{stretched} with the following property: Let $u \in H^1_{\loc}(\Rd^d)$ be $a$-harmonic in $B_R(x)$ for $R>0$ and $x \in \Rd^d$, \textit{i.e.}  suppose that $u$ satisfies 
\begin{align*}
-\nabla \cdot (a\nabla u)  &= 0 \quad \quad\text{in } \quad \, B_R(x).
\end{align*}
Then for any $r,R$ with $R\geq r \geq r^*(x,a)>0$ we have the estimate 
\begin{align}
\label{large_scale_mvp}
\fint_{B_r(x)} |\nabla u|^2 \, \dd y \lesssim_{d,\lambda} \fint_{B_R(x)} |\nabla u|^2 \, \dd y.
\end{align}
\end{theorem}
In our argument, we have used the following estimate for the first-order homogenization corrector that has been proven in \cite[Theorem~2]{GNO5} under the assumption that the ensemble satisfies a log-Sobolev inequality. The result can be found in the more general setting of homogenization for nonlinear uniformly elliptic systems under the assumption that the ensemble satisfies a spectral gap in \cite[Corollary~15]{FN_2020}. For the degenerate linear elliptic setting see \cite{Bella_Kniely}.
\begin{theorem}[Corrector estimates in stochastic homogenization]
\label{CorrectorBoundGNO4}
Under Assumptions (A1)--(A3) and for $d\geq 3$, there exists a random field $\C(a^\eps, x)$ with stretched exponential moments in the sense \eqref{stretched} such that 
\begin{align}
\label{GNO_bound_balls}
\sup_{r\geq \eps} \Big(\fint_{B_r(x)} |\phi^\eps|^2+|\sigma^\eps|^2 \, \dd y \Big)^{\frac{1}{2}} \leq 
\mathcal{C}(a,x)  \varepsilon
\end{align}
holds for any $x \in \Rd^d$.
\end{theorem}

Relying in addition on Assumption (A4), we also obtain small-scale regularity properties of the correctors.
\begin{lemma}[Regularity of the correctors on small scales]
\label{CorrectorRegularity}
Let Assumptions (A1)--(A4) be satisfied. Then for any $0<\gamma<\nu$ and any $x_0\in \mathbb{R}^d$ there exists a random constant $\mathcal{C}(a,x_0)$ with a uniform bound on suitable stretched exponential moments
\begin{align}
\label{ExpMoments}
\left \langle \exp\Big(\mathcal{C}(a,x_0)^{1/C}/C\Big) \right \rangle \leq 2
\end{align}
with $C$ depending possibly on $\gamma$ but not on $x_0$ such that the following is true: The estimates
\begin{align*}
|\nabla \phi_i^\eps(x)-\nabla \phi_i^\eps(y)|\leq \mathcal{C}(a,x_0) \frac{|x-y|^\gamma}{\varepsilon^\gamma}
\end{align*}
and
\begin{align*}
|\nabla \sigma_{ijk}^\eps(x)-\nabla \sigma_{ijk}^\eps(y)|\leq \mathcal{C}(a,x_0) \frac{|x-y|^\gamma}{\varepsilon^\gamma}
\end{align*}
hold for any $x,y\in B_\eps(x_0)$.

Furthermore, for any $0<\gamma<\nu$ and any $x_0 \in \overline{\domain}$ there exists a random constant $\mathcal{C}(a,x_0)$ with a uniform bound on suitable stretched exponential moments (in the sense \eqref{ExpMoments}) such that the following is true: The estimate
\begin{align*}
|\nabla \theta_i^\eps(x)-\nabla \theta_i^\eps(y)|\leq \mathcal{C}(a,x_0) \frac{|x-y|^\gamma}{\varepsilon^\gamma}
\end{align*}
holds for any $x,y\in \overline{\domain}\cap B_{\eps}(x_0)$.
\end{lemma}
\begin{proof}
Relying on Assumptions (A1) and (A4), for any $x_0\in \mathbb{R}^d$ there exists a small enough (random) radius $\rho(a^\eps,x_0)\in (0,\eps]$ with $\mathbb{E}[\exp((\eps/\rho(a^\eps,x_0))^{1/C}/C)] \leq 2$ such that by classical Schauder theory (see e.\,g.\ \cite{GM_book}) we have for any $a$-harmonic function $u$ in $B_\rho(x_0)$
\begin{align*}
||(\nabla u)(\tfrac{\cdot-x_0}{\eps})||_{C^\alpha(B_1)}
\leq C \Big(\fint_{B_{\rho(a^\eps,x_0)}(x_0)} |\nabla u|^2 \dx \Big)^{1/2}
\leq C \Big(\frac{\eps}{\rho(a^\eps,x_0)}\Big)^{d/2} \Big(\fint_{B_{\eps}(x_0)} |\nabla u|^2 \dx \Big)^{1/2}.
\end{align*}
Applying this bound to the $a$-harmonic function $x_i+\phi_i^\eps$ and using a covering argument to cover an $\eps$-ball and inserting the bound \eqref{GNO_bound_balls}, we arrive at our estimates for $\phi_i^\eps$. Given the regularity bound for $\phi_i^\eps$, the argument for $\sigma_{ijk}^\eps$ is similar, though we now need to account for a right-hand side. The bounds for $\theta_i^\eps$ are analogous to the case of $\phi_i^\eps$, possibly (depending on the location of $x_0$) also making use of boundary regularity via Schauder theory as well as also using the bound on $\eps\nabla \theta_i^\eps$ from Proposition~\ref{Prop_1_aux}.
\end{proof}

\section*{Funding}

The work of PB was supported by the Deutsche Forschungsgemeinschaft (DFG, German Research Foundation) under the priority program SPP 2256, Project number 441469601.
This research was funded in whole or in part by the Austrian Science Fund (FWF) ESP4053024. For open access purposes, the author has applied a CC BY public copyright license to any author-accepted manuscript version arising from this submission.
This project has received funding from the European Research Council (ERC) under the European Union's Horizon 2020 research and innovation programme (grant agreement No 948819). \includegraphics[height=3.0mm]{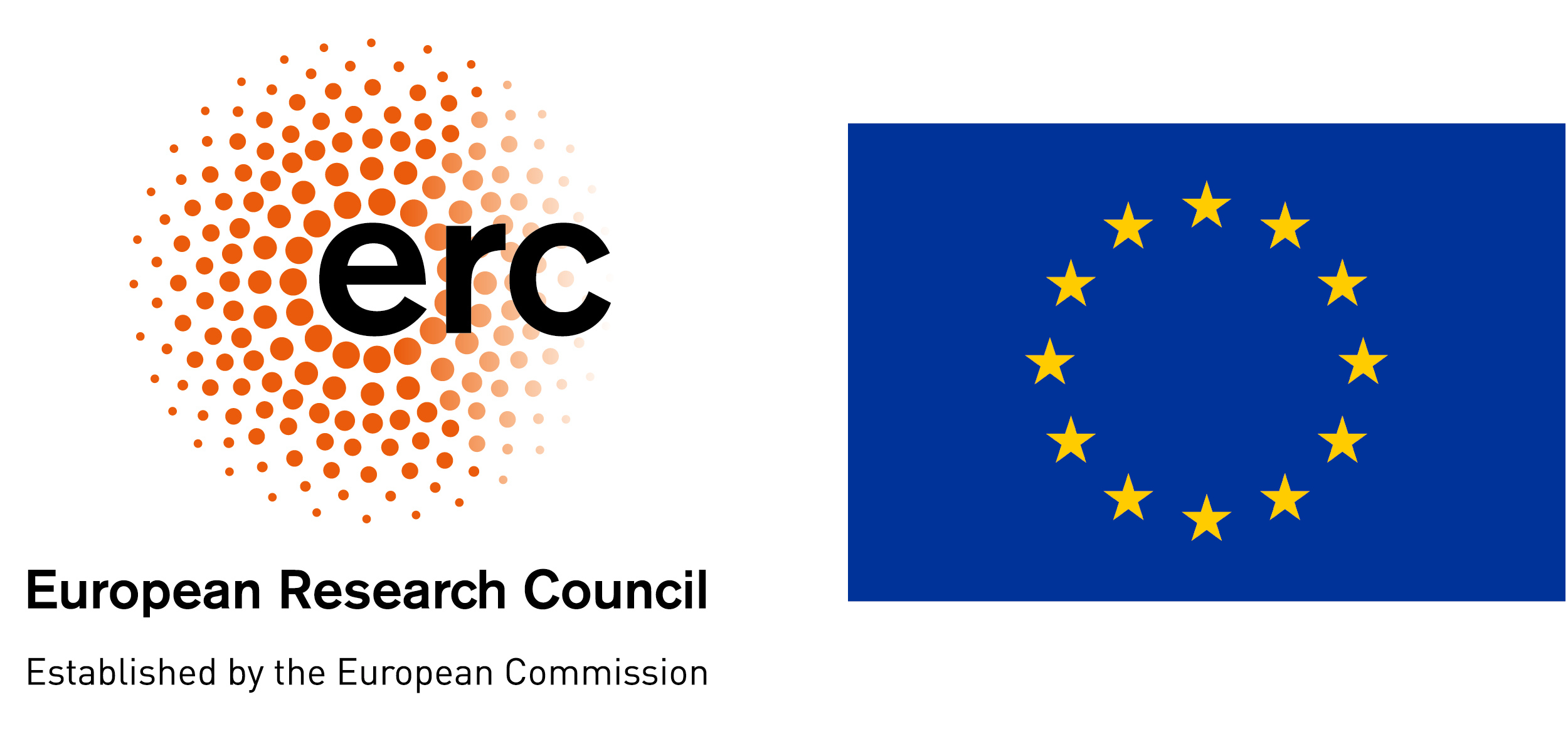}

\bibliographystyle{abbrv}
\bibliography{stochastic_homogenization}

\end{document}